\documentclass[11pt,reqno]{amsart}

\newtheorem{thm}{Theorem}[section]
\newtheorem*{thm*}{Theorem}
\newtheorem{lem}[thm]{Lemma}
\newtheorem*{lem*}{Lemma}

\newtheorem{claim}[thm]{Claim}

\newtheorem{prop}[thm]{Proposition}

\theoremstyle{definition}

\newtheorem{case}{Case}\renewcommand{\thecase}{}
\newtheorem*{case*}{Case}

\newtheorem*{defn*}{Definition}

\newtheorem*{exmp*}{Example}

\newtheorem{rmk}[thm]{Remark}
\newtheorem*{rmk*}{Remark}
\newtheorem{step}{Step}\renewcommand{\thestep}{}

\theoremstyle{remark}

\makeatletter
\def\alphenumi{
  \def\theenumi{\alph{enumi}}
  \def\p@enumi{\theenumi}
  \def\labelenumi{(\@alph\c@enumi)}}
\makeatother

\makeatletter
\def\thecase{\@arabic\c@case}
\makeatother

\makeatletter
\def\thestep{\@arabic\c@step}
\makeatother

\newcount\hh
\newcount\mm
\mm=\time
\hh=\time
\divide\hh by 60
\divide\mm by 60
\multiply\mm by 60
\mm=-\mm
\advance\mm by \time
\def\hhmm{\number\hh:\ifnum\mm<10{}0\fi\number\mm}

\let\oldmarginpar\marginpar
\renewcommand\marginpar[1]{\-\oldmarginpar[\raggedleft\footnotesize #1]%
{\raggedright\footnotesize #1}}

\newcommand\dotprod{\hbox{$\cdot$}}

\newcommand\EE{\mathbb{E}}

\newcommand\NN{\mathbb{N}}

\newcommand\PP{\mathbb{P}}

\newcommand\RR{\mathbb{R}}

\newcommand\cF{{\mathcal{F}}}

\newcommand\cH{{\mathcal{H}}}

\newcommand\cS{{\mathcal{S}}}
\newcommand\cT{{\mathcal{T}}}

\newcommand\fw{{\mathfrak{w}}}

\newcommand\eps{\varepsilon}

\newcommand\dist{\operatorname{dist}}

\newcommand\loc{\operatorname{loc}}

\numberwithin{equation}{section}

\usepackage{amssymb}
\usepackage{hyperref}
\usepackage{mathrsfs}
\usepackage{slashed}
\usepackage[usenames]{color}
\usepackage{url}
\usepackage{verbatim}

\hypersetup{pdftitle={The obstacle problem for the fractional Laplacian with drift}}

\hypersetup{pdfauthor={Arshak Petrosyan and Camelia A. Pop}}

\begin{document}

\title[Fractional Laplacian with drift]{Optimal regularity of solutions to the obstacle problem for the fractional Laplacian with drift}

\author[A. Petrosyan]{Arshak Petrosyan}
\address[AP]{Department of Mathematics, Purdue University, West Lafayette, IN 47907}
\email{arshak@math.purdue.edu}

\author[C. Pop]{Camelia A. Pop}
\address[CP]{Department of Mathematics, University of Pennsylvania, Philadelphia, PA 19104-6395}
\email{cpop@math.upenn.edu}

\date{\today{ }\hhmm}

\begin{abstract}
We prove existence, uniqueness and optimal regularity of solutions to the stationary obstacle problem defined by the fractional Laplacian operator with drift, in the subcritical regime. As in \cite{Caffarelli_Salsa_Silvestre_2008}, we localize our problem by considering a suitable extension operator introduced in \cite{Caffarelli_Silvestre_2007}. The structure of the extension equation is different from the one constructed in \cite{Caffarelli_Salsa_Silvestre_2008}, in that the obstacle function has less regularity, and exhibits some singularities. To take into account the new features of the problem, we prove a new monotonicity formula, which we then use to establish the optimal regularity of solutions.
\end{abstract}

%


\thanks{AP was supported in part by NSF grant DMS-1101139. CP would like to thank Paul Feehan for introducing her to obstacle problems.}

\maketitle

\tableofcontents

\section{Introduction}
\label{sec:Intro}
We consider the linear operator defined by the fractional Laplacian with drift,
\begin{equation}
\label{eq:Operator}
Lu(x) : = \left(-\Delta\right)^s u(x) + b(x)\dotprod\nabla u (x)+c(x) u(x),\quad \forall u \in C^{2}_c(\RR^n),
\end{equation}
where the coefficient functions $b:\RR^n\rightarrow\RR^n$ and $c:\RR^n\rightarrow\RR$ are assumed to be H\"older continuous.
The action of the fractional Laplacian operator on functions $u\in C^2_c(\RR^n)$ is given by the singular integral,
\begin{equation*}
(-\Delta)^s u(x) = c_{n,s} \hbox{ p.v.} \int_{\RR^n}\frac{u(x)-u(y)}{|x-y|^{n+2s}}\, dy,
\end{equation*}
understood in the sense of the principal value. The constant $c_{n,s}$ is positive and depends only on the dimension $n\in \NN$, and on the parameter $s\in (0,1)$. The range $(0,1)$ of the parameter $s$ is particularly interesting because in this case the fractional Laplacian operator is the infinitesimal generator of the symmetric $2s$-stable process \cite[Example 3.3.8]{Applebaum}.

The fractional Laplacian plays the same paradigmatic role in the theory of non-local operators that the Laplacian plays in the theory of local elliptic operators. For this reason, the regularity of solutions to equations defined by the fractional Laplacian and its gradient perturbation is intensely studied in the literature. In this article, we study the stationary obstacle problem defined by the fractional Laplacian operator with drift \eqref{eq:Operator}, in the subcritical regime, that is, the case when the parameter $s$ belongs to the range $(1/2,1)$. Given an obstacle function, $\varphi\in C^{3s}(\RR^n)\cap C_0(\RR^n)$, we prove existence, uniqueness and optimal regularity of solutions in H\"older spaces, $u \in C^{1+s}(\RR^n)$, for the stationary obstacle problem,
\begin{equation}
\label{eq:Obstacle_problem}
\min\{\left(-\Delta\right)^s u(x) + b(x)\dotprod\nabla u (x)+c(x) u(x), u(x)-\varphi(x)\}=0,\quad\forall x\in \RR^n.
\end{equation}
Our main result is
\begin{thm}[Existence, uniqueness and optimal regularity of solutions]
\label{thm:Solutions}
Let $s \in (1/2,1)$. Assume that $b \in C^{s}(\RR^n;\RR^n)$, and $c \in C^{s}(\RR^n)$ is such that 
\begin{equation}
\label{eq:Nonnegative_lower_bound_c}
c \geq 0\quad\hbox{on }\RR^n.
\end{equation}
Assume that $\varphi\in C^{3s}(\RR^n)\cap C_0(\RR^n)$ is such that 
\begin{equation}
\label{eq:Boundedness_L_phi_positive_part}
(L\varphi)^+\in L^{\infty}(\RR^n).
\end{equation}
Then the obstacle problem \eqref{eq:Obstacle_problem} has a solution, $u \in C^{1+s}(\RR^n)$. If in addition the vector field $b:\RR^n\rightarrow\RR^n$ is a Lipschitz function, and there is a positive constant, $c_0$, such that
\begin{equation}
\label{eq:Lower_bound_c}
c(x)\geq c_0,\quad\forall x\in \RR^n,
\end{equation}
then there is a unique solution, $u \in C^{1+s}(\RR^n)$, to the obstacle problem \eqref{eq:Obstacle_problem}.
\end{thm}

We note that the properties of the fractional Laplacian operator with drift differ substantially depending whether the parameter $s$ takes values in the range $(0,1/2)$ (the supercritical regime), is equal to $1/2$ (the critical regime), or takes values in $(1/2,1)$ (the subcritical regime) \cite[\S 1]{Silvestre_2012a, Caffarelli_Vasseur_2010}. In the critical and subcritical regime, the fractional Laplacian operator with drift defines an elliptic pseudodifferential operator in the sense of \cite[\S 3.9]{Taylor_vol1}, and so, the drift component can be treated as a lower order term, a fact that we use extensively in our analysis of the obstacle problem \eqref{eq:Obstacle_problem}. In the supercritical regime, the operator $L$ is no longer elliptic and our analysis no longer applies. A study of the regularity of solutions in Sobolev spaces and of the Green's kernel of the stationary \emph{linear} equation defined by the fractional Laplacian with drift in the supercritical regime, can be found in \cite{Epstein_Pop_2013}.

The stationary obstacle problem defined by the fractional Laplacian operator \emph{without drift} was studied by Silvestre \cite{Silvestre_2007}, and by Caffarelli, Salsa and Silvestre \cite{Caffarelli_Salsa_Silvestre_2008}. In \cite{Silvestre_2007}, it is established the almost optimal regularity of solutions to the obstacle problem for the fractional Laplacian without drift \cite[Theorem 5.8]{Silvestre_2007}, that is, given an obstacle function, $\varphi\in C^{1+\beta}(\RR^n)$, the solution is shown to be $C^{1+\alpha}(\RR^n)$, for all $\alpha \in (0,\beta\wedge s)$. In \cite{Caffarelli_Salsa_Silvestre_2008}, the authors prove the optimal regularity of solutions \cite[Corollary 6.8]{Caffarelli_Salsa_Silvestre_2008}, that is, the solution $u$ belongs to $C^{1+s}(\RR^n)$, when the obstacle function, $\varphi$, is assumed to be in $C^{2,1}(\RR^n)$, and they establish the $C^{1,\alpha}$ regularity of the free boundary in a neighborhood of regular points \cite[Theorem 7.7]{Caffarelli_Salsa_Silvestre_2008}. 

In our work, we prove existence, uniqueness and optimal regularity of solutions to the stationary obstacle problem defined by the fractional Laplacian \emph{with drift}. In proving the existence and uniqueness of solutions, we take a different approach than in \cite{Silvestre_2007, Caffarelli_Salsa_Silvestre_2008}. Specifically, to obtain existence of solutions to the obstacle problem, we first study the linear and the penalized problem defined the fractional Laplacian with drift, which we solve by establishing a priori Schauder estimates and applying perturbation arguments. To obtain uniqueness of solutions, we use a probabilistic approach and we establish the stochastic representation of solutions. The stochastic representation of solutions is especially important in mathematical finance, where the value function of American-style options are given in the form of a stochastic representation. 

Our strategy in proving the optimal regularity of solutions is similar to that of \cite{Caffarelli_Salsa_Silvestre_2008}, but there are certain aspects in which the method of \cite{Caffarelli_Salsa_Silvestre_2008} is not applicable to our framework, which we now outline. In \S \ref{sec:Solutions_optimal_regularity}, we show that the obstacle problem \eqref{eq:Obstacle_problem} for the fractional Laplacian with drift can be reduced to one without drift, in which the obstacle function can only be assumed to belong to the space of functions $C^{2s+\alpha}(\RR^n)$, for all $\alpha\in (0,s)$, while in \cite{Caffarelli_Salsa_Silvestre_2008}, the obstacle function is assumed to have better regularity, i.e., it belongs to $C^{2,1}(\RR^n)$. Similarly to \cite{Caffarelli_Salsa_Silvestre_2008}, we construct a new monotonicity formula of Almgren type in Proposition \ref{prop:Monotonicity_formula}, which takes into account the limitation in regularity of the obstacle function. We then consider a suitable sequence of rescaled functions, for which we prove uniform estimates in H\"older spaces, and which we use together with the monotonicity formula, to establish the optimal regularity of solutions to the obstacle problem \eqref{eq:Obstacle_problem}. The arguments employed in \cite{Caffarelli_Salsa_Silvestre_2008} to establish the uniform estimates in H\"older spaces of the sequence of rescalings \cite[Proposition 4.3]{Caffarelli_Salsa_Silvestre_2008}, and to obtain the growth of the solution in a neighborhood of a free boundary point \cite[Lemma 6.5]{Caffarelli_Salsa_Silvestre_2008}, are not applicable to our case, due to the presence of a singular measure in the structure of our problem. Instead, our method of proof is based on a suitable application of the Moser iteration technique to obtain supremum and growth estimates (Lemma \ref{lem:Uniform_boundedness_rescalings} and Proposition \ref{prop:Growth_v_around_0}), and of a localization procedure described in \cite[Theorem 8.11.1]{Krylov_LecturesHolder} to obtain Schauder estimates (Lemma \ref{lem:Uniform_boundedness_C_1_alpha_n_ball_rescalings}). While Krylov uses the method described in \cite[Theorem 8.11.1]{Krylov_LecturesHolder} to obtain a priori Schauder estimates for solutions to a \emph{linear equation}, we apply it to obtain estimates for solutions to an \emph{obstacle problem}.

\subsection{Comparison with previous research}
\label{subsec:Previous_research}
We may compare our result on existence, uniqueness and almost optimal regularity of solutions to the stationary obstacle problem \eqref{eq:Obstacle_problem}, Proposition \ref{prop:Solutions_partial_regularity}, with the analogous results obtained by L. Silvestre, in \cite{Silvestre_2007}, for the obstacle problem defined by the fractional Laplacian operator \emph{without drift}, \cite[Complementarity conditions (1.1) and (1.2)]{Silvestre_2007}. Our results are proved in the subcritical regime, that is, the case when the parameter $s$ is contained in the range $(1/2,1)$, while the results of \cite{Silvestre_2007} hold for all $s\in (0,1)$, but the operator does not contain a drift component. The assumption used in our article that $s\in (1/2,1)$ plays an important role because it allows us to treat the drift component as a lower order term of the operator $L$. In \cite[\S 1.1]{Silvestre_2007}, existence of solutions in Sobolev spaces can be immediately obtained by variational methods, which do not seem readily applicable to our framework due to the presence of the lower order term. Instead, in \S \ref{sec:Solutions_partial_regularity}, we prove a priori Schauder estimates and we use a perturbation argument to obtain the existence of solutions in the H\"older space $C^{1+\alpha}(\RR^n)$, for some $\alpha\in (0,s)$. To improve the regularity of solutions, we begin with the fact that we know that the solutions to the obstacle problem are in $C^{1+\alpha}(\RR^n)$, and we use a bootstrap argument in conjunction to the almost optimal regularity of solutions proved in \cite[Theorem 5.8]{Silvestre_2007}, to establish  in Proposition \ref{prop:Solutions_partial_regularity} the almost optimal regularity of solutions to our obstacle problem \eqref{eq:Obstacle_problem}. 

We next compare our work with that of L. Caffarelli, S. Salsa and L. Silvestre \cite{Caffarelli_Salsa_Silvestre_2008}, who establish the optimal regularity of solutions to the obstacle problem defined by the fractional Laplacian operator \emph{without drift}, when the obstacle problem is assumed to belong to the space of functions $C^{2,1}(\RR^n)$ (see \cite[Corollary 6.8]{Caffarelli_Salsa_Silvestre_2008}). In \S \ref{sec:Monotonicity_formula}, we reduce our obstacle problem \eqref{eq:Obstacle_problem} to one without drift, but the obstacle function can at most be assumed to belong to $C^{2s+\alpha}(\RR^n)$, for all $\alpha\in (0,s)$, due to the presence of the drift component in the definition \eqref{eq:Operator} of the operator $L$. Analogous to \cite{Caffarelli_Salsa_Silvestre_2008}, we introduce an auxiliary `height' function, $v$, in \eqref{eq:Auxiliary_function_v}, and we consider a suitable $L_a$-extension of our nonlocal problem to a local one which satisfies conditions \eqref{eq:Upper_bound_L_a} and \eqref{eq:Equality_L_a}. Compared with \cite[Conditions (2.2)-(2.5)]{Caffarelli_Salsa_Silvestre_2008}, our extended problem has the property that $L_a v$ contains a singular measure on the set $\{y=0\}\backslash\{v=0\}$, which is not the case in \cite{Caffarelli_Salsa_Silvestre_2008}. This limitation in the regularity of the obstacle function, and the singularity appearing in our extended problem make many of the arguments used in \cite{Caffarelli_Salsa_Silvestre_2008} inapplicable to our framework. We prove a new monotonicity formula of Almgren type in \S \ref{sec:Monotonicity_formula}, and we replace the comparison arguments in \cite{Caffarelli_Salsa_Silvestre_2008} by adapting the Moser iterations method to our framework, in Lemma \ref{lem:Uniform_boundedness_rescalings} and Proposition \ref{prop:Growth_v_around_0}, and by applying the localization method of \cite[Theorem 8.11.1]{Krylov_LecturesHolder} to prove uniform Schauder estimates in Lemma \ref{lem:Uniform_boundedness_C_1_alpha_n_ball_rescalings}. A key role in our analysis plays the result of Proposition \ref{prop:Phi_at_0}, in which it is established a lower bound of the function $\Phi_v^p(r)$ defined in \eqref{eq:Phi}. Even though the definition of our function $\Phi_v^p(r)$ differs from its analogue in \cite[\S 3]{Caffarelli_Salsa_Silvestre_2008}, indirectly we appeal to the lower bounds established in \cite[Lemma 6.1]{Caffarelli_Salsa_Silvestre_2008} to derive our Proposition \ref{prop:Phi_at_0}.

\subsection{Outline of the article}
\label{subsec:Outline}
Our main result, Theorem \ref{thm:Solutions}, is proved in \S \ref{sec:Solutions}, which is organized in 
four subsections. We being in \S \ref{sec:Solutions_partial_regularity} by establishing the existence and uniqueness of solutions in H\"older spaces to the linear equation \eqref{eq:Linear_equation} defined by the operator $L$ (Lemma \ref{lem:Existence_uniqueness_linear_equation}). We use this result to solve the penalized equation (Lemma \ref{lem:Existence_penalized_equation}), which leads to the proof of the existence of solutions to the obstacle problem \eqref{eq:Obstacle_problem} (Proposition \ref{prop:Existence_Holder_obstacle_problem}). Via a bootstrap argument, we prove that the solutions constructed in Proposition \ref{prop:Existence_Holder_obstacle_problem} have the almost optimal regularity, that is, they belong to the space of functions $C^{1+\alpha}(\RR^n)$, for all $\alpha\in(0,s)$, when we assume that the obstacle function is contained in $C^{1+s}(\RR^n)\cap C_0(\RR^n)$ (Proposition \ref{prop:Solutions_partial_regularity}). In \S \ref{sec:Uniqueness}, we prove the uniqueness of solutions of the obstacle problem \eqref{eq:Obstacle_problem} by establishing that they admit a suitable stochastic representation (Proposition \ref{prop:Uniqueness}). We show how to reduce the obstacle problem for the fractional Laplacian with drift \eqref{eq:Obstacle_problem} to one \emph{without} drift in \S \ref{sec:Monotonicity_formula}. Using the extended problem introduced in \cite{Caffarelli_Silvestre_2007}, we introduce a suitable `height function', $v$, in \eqref{eq:Auxiliary_function_v}, and we prove a monotonicity formula (Propositions \ref{prop:Monotonicity_formula} and \ref{prop:Phi_at_0}), which we then use in \S \ref{sec:Solutions_optimal_regularity} to establish the growth of the function $v$ is a neighborhood of a free boundary point (Proposition \ref{prop:Growth_v_around_0}). Finally, we give the proof of the optimal regularity of solutions (Theorem \ref{thm:Solutions}). In \S \ref{sec:Auxiliary_results_proofs}, we give the proof of a series of auxiliary results, and of Proposition \ref{prop:Monotonicity_formula}. In \S \ref{subsec:Function_spaces}, we state the definitions of the spaces of functions, and in \S \ref{subsec:Notation}, we introduce the notations and conventions we use throughout our article.

\subsection{Function spaces}
\label{subsec:Function_spaces}
Let $k, m$ and $n$ be positive integers, and let $\Omega\subseteq\RR^n$ be an open set. We denote by $C^{\infty}_c(\Omega;\RR^m)$ the space of smooth functions, $u:\Omega\rightarrow\RR^m$, with compact support in $\Omega$. The space $C^k(\bar\Omega;\RR^m)$ consists of functions, $u:\bar\Omega\rightarrow\RR^m$, which admit derivatives up to order $k$, such that $u$ and its derivatives up to order $k$ are continuous and bounded on $\bar\Omega$. The space $C^k(\bar\Omega;\RR^m)$ endowed with the norm
$$
\|u\|_{C^k(\bar\Omega;\RR^m)} = \sup_{\stackrel{\beta\in\NN^n}{|\beta| \leq k}}\sup_{x\in \bar\Omega} |D^{\beta}u(x)|<+\infty,\quad\forall u \in C^k(\bar\Omega;\RR^m),
$$
is a Banach space. In the preceding definition, for all multi-indices $\beta\in\NN^n$, we let $|\beta|$ denote the sum of its components. 

Let $\alpha \in (0,1)$. The H\"older space $C^{k+\alpha}(\bar\Omega;\RR^m)$ consists of function $u\in C^k(\bar\Omega;\RR^m)$ satisfying the property that the seminorm
$$
\left[u\right]_{C^{k+\alpha}(\bar\Omega;\RR^m)}
:= \sup_{\stackrel{\beta\in\NN^n}{|\beta| \leq k}}
\sup_{\stackrel{x, y \in \bar\Omega}{x \neq y}} \frac{|D^{\beta} u(x)-D^{\beta} u(y)|}{|x-y|^{\alpha}} < +\infty,
\quad\forall u \in C^{k+\alpha}(\bar\Omega;\RR^m).
$$
The space $C^{k+\alpha}(\bar\Omega;\RR^m)$ endowed with the norm
$$
\|u\|_{C^{k+\alpha}(\bar\Omega;\RR^m)} = \|u\|_{C^k(\bar\Omega;\RR^m)}+ \left[u\right]_{C^{k+\alpha}(\bar\Omega;\RR^m)},
\quad\forall u \in C^{k+\alpha}(\bar\Omega;\RR^m),
$$
is a Banach space. We let $C^{k+\alpha}_0(\bar\Omega;\RR^m)$ be the closure of the space $C^{\infty}_c(\Omega;\RR^m)$ with respect to the norm $\|\cdot\|_{C^{k+\alpha}(\bar\Omega;\RR^m)}$. As usual, the space $C^{k+\alpha}_0(\bar\Omega;\RR^m)$ endowed with the norm $\|\cdot\|_{C^{k+\alpha}(\bar\Omega;\RR^m)}$ is a Banach space. The spaces $C^k_{\loc}(\RR^n;\RR^m)$ and $C^{k+\alpha}_{\loc}(\RR^n;\RR^m)$ consists of functions, $u:\RR^n\rightarrow\RR^m$, which belong to $C^k(K;\RR^m)$ and $C^{k+\alpha}(K;\RR^m)$, respectively, for all compact sets $K\subset\RR^n$. When $k=0$, we omit the superscript $k$ from the notation of the space $C^k(\bar\Omega;\RR^m)$, and when $m=1$, we write $C^k(\bar\Omega)$ instead of $C^k(\bar\Omega;\RR)$. The analogous convention applies to the spaces $C^{k+\alpha}(\bar\Omega;\RR^m)$, $C^k_{\loc}(\RR^n;\RR^m)$, and $C^{k+\alpha}_{\loc}(\RR^n;\RR^m)$.

\subsection{Notations and conventions}
\label{subsec:Notation}
We let $\cS(\RR^n)$ denote the Schwartz space \cite[Definition (3.3.3)]{Taylor_vol1} consisting of smooth functions whose
derivatives of all orders decrease faster than any polynomial at infinity, and we let $\cS'(\RR^n)$ denote its dual space, the space of tempered distributions \cite[\S 3.4]{Taylor_vol1}. We adopt the following definition of the Fourier transform  of a function $u\in\cS(\RR^n)$,
$$
\widehat u (\xi) = \int_{\RR^n} e^{-ix\dotprod\xi} u(x)\ dx,\quad\forall \xi\in\RR^n.
$$
Given real numbers, $a$ and $b$, we let $a \wedge b:=\min\{a,b\}$ and $a \vee b:=\max\{a,b\}$. If $v,w\in\RR^n$, we denote by $v\cdot w$ their scalar product. For $x_0\in\RR^{n+1}$ and $r>0$, let $B_r(x_0)$ be the Euclidean ball in $\RR^{n+1}$ of radius $r$ centered at $x_0$, and for $x_0\in\RR^n$ and $r>0$, let $B'_r(x_0)$ be the Euclidean ball in $\RR^n$ of radius $r$ centered at $x_0$. We denote by $B^+_r(x^0)$ the half-ball, $B_r(x^0)\cap\left(\RR^{n}\times\RR_+\right)$, where $\RR_+:=(0,\infty)$. For brevity, when $x_0=O$, we write $B_r$, $B'_r$ and $B^+_r$ instead of $B_r(O)$, $B'_r(O)$ and $B^+_r(O)$, respectively. 

For a set $S\subseteq\RR^n$, we denote its complement by $S^c:=\RR^n\backslash S$, and we let $\hbox{int } S$ denote its topological interior.

\section{Existence and optimal regularity of solutions}
\label{sec:Solutions}
In this section we prove the main result of our article, Theorem \ref{thm:Solutions}. We organize its content into three parts. In \S \ref{sec:Solutions_partial_regularity}, we prove the existence of solutions in H\"older spaces to the obstacle problem \eqref{eq:Obstacle_problem}, and we show that the solutions have the almost optimal regularity. In \S \ref{sec:Uniqueness}, we give sufficient conditions which ensure that the obstacle problem \eqref{eq:Obstacle_problem} has a unique solution in $C^{1+\alpha}(\RR^n)$, where $\alpha\in ((2s-1)\vee 0,1)$. We prove the uniqueness of solutions by establishing their stochastic representation. Even though in general we assume throughout our article that $s\in (1/2,1)$, the stochastic representation of solutions holds for all $s\in (0,1)$. In \S \ref{sec:Monotonicity_formula}, we prove a version of the monotonicity formula suitable for our operator, which is then used in \S \ref{sec:Solutions_optimal_regularity} to obtain the optimal regularity of solutions to the obstacle problem \eqref{eq:Obstacle_problem}.

\subsection{Existence and almost optimal regularity of solutions}
\label{sec:Solutions_partial_regularity}

In this section, we prove existence of solutions, $u$, to the obstacle problem \eqref{eq:Obstacle_problem} having almost optimal regularity, that is, $u \in C^{1+\beta}(\RR^n)$, for all $\beta \in (0,s)$, when the obstacle function, $\varphi$, is assumed to belong to the space of functions $C^{1+s}(\RR^n)\cap C_0(\RR^n)$. The main result of this section is

\begin{prop}[Existence, uniqueness and almost optimal regularity of solutions]
\label{prop:Solutions_partial_regularity}
Let $s \in (1/2,1)$ and $\alpha \in (0,1)$. Assume that the coefficient function $b \in C^{s}(\RR^n;\RR^n)$, and $c \in C^{s}(\RR^n)$ and satisfies condition \eqref{eq:Nonnegative_lower_bound_c}. Assume that the obstacle function $\varphi\in C^{1+\alpha}(\RR^n)\cap C_0(\RR^n)$ and satisfies condition \eqref{eq:Boundedness_L_phi_positive_part}. Then the obstacle problem \eqref{eq:Obstacle_problem} defined by the fractional Laplacian with drift has a solution, $u \in C^{1+\beta}(\RR^n)$, for all $\beta<\alpha\wedge s$.
\end{prop}

To prove Proposition \ref{prop:Solutions_partial_regularity}, we use a series of preliminary results. We first prove a maximum principle (Lemma \ref{lem:Comparison_principle}) which is used to obtain the existence and uniqueness of solutions in H\"older spaces to the linear equation \eqref{eq:Linear_equation} defined by the fractional Laplacian with drift (Lemma \ref{lem:Existence_uniqueness_linear_equation}). This result is applied to prove existence of solutions to the penalized equation \eqref{eq:Penalized_equation} defined by the fractional Laplacian with drift, which gives us the existence of solutions to the obstacle problem (Proposition \ref{prop:Existence_Holder_obstacle_problem}). The solutions we obtain at this point have less regularity than the one stated in the conclusion of Proposition \ref{prop:Solutions_partial_regularity}. A bootstrap argument is then used, together with Proposition \ref{prop:Existence_Holder_obstacle_problem}, to give the proof of Proposition \ref{prop:Solutions_partial_regularity}.

\subsubsection{The linear equation defined by the fractional Laplacian with drift}
\label{subsec:Linear_equation}
In this section we establish the existence and uniqueness in H\"older spaces of solutions to the linear equation defined by the fractional Laplacian with drift, 
\begin{equation}
\label{eq:Linear_equation}
Lu=f\quad\hbox{on } \RR^n,
\end{equation}
in the case when $s\in (1/2,1)$. 

\begin{lem}[Existence and uniqueness of solutions to the linear equation]
\label{lem:Existence_uniqueness_linear_equation}
Let $s\in(1/2,1)$. Let $\alpha \in (0,1)$ be such that $2s+\alpha$ is not an integer. Assume that the coefficient functions $b\in C^{\alpha}(\RR^n;\RR^n)$, and $c \in C^{\alpha}(\RR^n)$ and satisfies condition \eqref{eq:Lower_bound_c}. Then, there is a positive constant, $C=C(\alpha, \|b\|_{C^{\alpha}(\RR^n;\RR^n)}, \|c\|_{C^{\alpha}(\RR^n)}, c_0, n, s)$, such that for any source function, $f \in C^{\alpha}(\RR^n)$, there is a unique solution, $u \in C^{2s+\alpha}(\RR^n)$, to the linear equation \eqref{eq:Linear_equation}, and the function $u$ satisfies the Schauder estimate,
\begin{equation}
\label{eq:Holder_estimate}
\|u\|_{C^{2s+\alpha}(\RR^n)} \leq C\|f\|_{C^{\alpha}(\RR^n)}.
\end{equation}
\end{lem}

\begin{rmk}[Regularity of solutions to the linear equation]
The regularity of solutions in Sobolev spaces, as opposed to H\"older spaces, to the linear equation defined by the fractional Laplacian with drift in the suprecritical regime, that is the case when $s\in (0,1/2)$, has been established in \cite{Epstein_Pop_2013}, using methods specific to the theory of pseudodifferential operators. We remark that the supercritical case is more difficult to treat than the subcritical regime, $s\in (1/2,1)$, because the operator is \emph{not elliptic}. In the subcritical case, the diffusion component dominates the drift term, and so, the drift term can be treated as a lower-order perturbation. This is an important fact that we use in the proof of Lemma \ref{lem:Existence_uniqueness_linear_equation}, but which cannot be extended to the supercritical regime.
\end{rmk}

To prove Lemma \ref{lem:Existence_uniqueness_linear_equation}, we commence with

\begin{lem}[Comparison principle]
\label{lem:Comparison_principle}
Let $s\in (1/2,1)$. Assume that the coefficient function $b\in C(\RR^n;\RR^n)$, and $c\in C_{\loc}(\RR^n)$ and satisfies condition \eqref{eq:Lower_bound_c}. If $u \in C(\RR^n)\cap C^1_{\loc}(\RR^n)$ satisfies
\begin{equation}
\label{eq:Lu_nonnegative}
\left(-\Delta\right)^s u + b\dotprod\nabla u +c u \geq 0\quad\hbox{on }\RR^n,
\end{equation}
then
\begin{equation}
\label{eq:u_nonnegative}
u \geq 0 \quad\hbox{on }\RR^n.
\end{equation}
\end{lem}

\begin{proof}
We consider the auxiliary function,
\begin{equation}
\label{eq:Auxiliary_function}
v(x):=(a+|x|^2)^p,\quad\forall x\in\RR^n,
\end{equation}
where $p$ is a fixed number in the interval $(0,1/2)$, and the positive constant $a$ will be suitably chosen below. Direct calculations give us, for all $x\in\RR^n$ and $i,j=1,\ldots,n$, 
\begin{equation}
\label{eq:Derivatives_aux_function}
\begin{aligned}
v_{x_i}(x) &= 2p(a+|x|^2)^{p-1}x_i,\\
v_{x_ix_j} &= 4p(p-1)(a+|x|^2)^{p-1}x_ix_j + 2p(a+|x|^2)^{p-1}\delta_{ij},
\end{aligned}
\end{equation}
where $\delta_{ij}$ denoted the Kronecker delta symbol. Because $p\in (0,1/2)$ and $a>0$, we see that the derivatives $v_{x_i}$ and $v_{x_ix_j}$ belong to $C(\RR^n)$, and so, we have that $v_{x_i}\in C^{\beta}(\RR^n)$, for all $\beta\in(0,1)$ and for all $i=1,\ldots,n$. We can find a positive constant, $C=C(p)$, such that
$$
\|v_{x_i}\|_{C^{\beta}(\RR^n)} \leq C,\quad\forall a\geq 1,\quad\forall \beta\in (0,1).
$$
By choosing $\beta>2s-1$, we have that
\begin{equation}
\label{eq:Upper_bound_frac_lap_u}
\begin{aligned}
|(-\Delta)^s v(x)| &\leq \left|\int_{B'_1}\frac{v(x+y)-v(x)-\nabla v(x)\cdot y}{|y|^{n+2s}} \ dy \right|
+ \left|\int_{(B'_1)^c}\frac{v(x+y)-v(x)-\nabla v(x)\cdot y}{|y|^{n+2s}} \ dy \right|\\
&\leq C\int_{B'_1}\frac{|y|^{1+\beta}}{|y|^{n+2s}} \ dy + C\int_{(B'_1)^c}\frac{|y|}{|y|^{n+2s}} \ dy.
\end{aligned}
\end{equation}
The first and second integral on the right-hand side of \eqref{eq:Upper_bound_frac_lap_u} are finite because we have chosen $\beta\in (2s-1,1)$ and $s\in (1/2,1)$. We obtain that there is a positive constant, $C_0$, such that
\begin{align*}
|(-\Delta)^s v(x)|&\leq C_0,\quad\forall x\in \RR^n.
\end{align*}
Using the preceding inequality, identities \eqref{eq:Auxiliary_function} and \eqref{eq:Derivatives_aux_function}, and condition \eqref{eq:Lower_bound_c}, we have that
\begin{align*}
Lv &\geq -C_0 +\frac{2pb(x)\dotprod x+c_0(a+|x^2|)}{(a+|x|^2)^{1-p}}.
\end{align*}
By choosing $a=a(\|b\|_{L^{\infty}(\RR^n)}, c_0)\geq 1$ sufficiently large, we can ensure that $Lv >0$ on $\RR^n$. For $\eps>0$, we consider the auxiliary function,
$$
w_{\eps}:=u+\eps v.
$$
Then $Lw_{\eps} >0$ on $\RR^n$, and we see that the function $w_{\eps}$ tends to $ \infty$, as $|x|\rightarrow\infty$, by definition \eqref{eq:Auxiliary_function} of $v$ and the fact that $u$ is a bounded function. If $w_{\eps}$ is not nonnegative, there is a point $x_0\in\RR^n$ where the function $w_{\eps}$ attains a global minimum. We have that
$$
w_{\eps}(x_0)<0,\ \nabla w_{\eps}(x_0)=0,\quad\hbox{and}\quad (-\Delta)^sw_{\eps}(x_0)<0.
$$
Because $c \geq 0$ on $\RR^n$, we have that $L w_{\eps}(x_0)<0$, which contradicts the fact that $Lw_{\eps} >0$ on $\RR^n$. Therefore, $w_{\eps} \geq 0$ on $\RR^n$, for all $\eps \geq 0$, and so, we obtain inequality \eqref{eq:u_nonnegative} by letting $\eps$ tend to zero.
\end{proof}

The following supremum estimate is a consequence of Lemma \ref{lem:Comparison_principle}.

\begin{lem}[Supremum estimate]
\label{lem:Sup_estimate}
Assume that the hypotheses of Lemma \ref{lem:Comparison_principle} hold. If $f \in C(\RR^n)$ and $u \in C(\RR^n)\cap C^1_{\loc}(\RR^n)$ is a solution to problem \eqref{eq:Linear_equation}, then $u$ satisfies
\begin{equation}
\label{eq:Sup_estimate}
\|u\|_{C(\RR^n)} \leq \frac{1}{c_0}\|f\|_{C(\RR^n)}.
\end{equation}
\end{lem}

\begin{proof}
Estimate \eqref{eq:Sup_estimate} follows from the observation that
$$
L\left(\pm u+\frac{1}{c_0}\|f\|_{C(\RR^n)}\right) =\pm f+\|f\|_{C(\RR^n)} \geq 0,
$$
and an application of Lemma \ref{lem:Comparison_principle}.
\end{proof}

We have the following a priori estimates in H\"older spaces.
\begin{lem}[A priori Schauder estimates]
\label{lem:Schauder_estimate}
Let $s\in(1/2,1)$. Let $\alpha \in (0,1)$ be such that $2s+\alpha$ is not an integer. Assume that the coefficient functions $b\in C^{\alpha}(\RR^n;\RR^n)$ and $c \in C^{\alpha}(\RR^n)$. Then, there is a positive constant, $C=C(\alpha, \|b\|_{C^{\alpha}(\RR^n;\RR^n)}, \|c\|_{C^{\alpha}(\RR^n)}, c_0, n, s)$, such that for any source function, $f \in C^{\alpha}(\RR^n)$, and any solution, $u \in C^{2s+\alpha}(\RR^n)$, to the linear equation \eqref{eq:Linear_equation}, the function $u$ satisfies the estimate
\begin{equation}
\label{eq:Holder_estimate_prim}
\|u\|_{C^{2s+\alpha}(\RR^n)} \leq C\left(\|f\|_{C^{\alpha}(\RR^n)}+\|u\|_{C(\RR^n)}\right).
\end{equation}
\end{lem}

\begin{proof}
By \cite[Proposition 2.8]{Silvestre_2007}, we obtain that
$$
\|u\|_{C^{2s+\alpha}(\RR^n)} \leq C\left(\|(-\Delta)^s u\|_{C^{\alpha}(\RR^n)}+\|u\|_{C(\RR^n)}\right),
$$
where $C=C(\alpha,n,s)$ is a positive constant. Using the fact that the coefficient functions $b \in C^{\alpha}(\RR^n;\RR^n)$ and $c\in C^{\alpha}(\RR^n)$, and the Interpolation inequalities \cite[Theorems 3.2.1 \& 8.8.1]{Krylov_LecturesHolder} together with the fact that $2s>1$, we obtain that, for any $\eps>0$, there is a positive constant, $C=C(\alpha, \|b\|_{C^{\alpha}(\RR^n;\RR^n)}, \|c\|_{C^{\alpha}(\RR^n)}, c_0, \eps, n, s)$, such that
$$
\|u\|_{C^{2s+\alpha}(\RR^n)} \leq \eps \|u\|_{C^{2s+\alpha}(\RR^n)} +C\left(\|L u\|_{C^{\alpha}(\RR^n)}+\|u\|_{C(\RR^n)}\right).
$$
Choosing $\eps=1/2$, we obtain the a priori Schauder estimate \eqref{eq:Holder_estimate_prim}.
\end{proof}

We can now prove the existence and uniqueness of solutions in H\"older spaces to the linear equation \eqref{eq:Linear_equation} defined by the fractional Laplacian with drift.

\begin{proof}[Proof of Lemma \ref{lem:Existence_uniqueness_linear_equation}]
Uniqueness of solutions follows from Lemma \ref{lem:Sup_estimate}. We first assume that the function $f$ is in $C^{\infty}_c(\RR^n)$, and prove the existence of solutions, $u \in C^{2s+\alpha}_0(\RR^n)$, to the simpler equation,
\begin{equation}
\label{eq:L_0}
L_0 u = (-\Delta)^s u +c_0 u =f\quad\hbox{on } \RR^n.
\end{equation}
Taking the Fourier transform in equation \eqref{eq:L_0}, and using the fact that
$$
\widehat{(\Delta)^s v}(\xi) = |\xi|^{2s} \widehat v(\xi),\quad\forall \xi\in\RR^n, \quad\forall v \in \cS(\RR^n),
$$
we set
$$
u(x):=(2\pi)^{-n} \int_{\RR^n}e^{i\xi x}\frac{1}{|\xi|^{2s}+c_0}\widehat f(\xi)\ d\xi,\quad\forall x\in\RR^n.
$$
We want to prove that $u \in C^{\infty}_0(\RR^n)$, and that $u$ solves equation \eqref{eq:L_0}. Because $f \in C^{\infty}_c(\RR^n)$, we have that $\widehat f \in \cS(\RR^n)$, and so $1/(|\xi|^{2s}+c_0)\widehat f(\xi) \in L^1(\RR^n)$. The Riemann-Lebesgue lemma \cite[Theorem 8.22 f]{Folland_realanalysis} shows that $u\in C_0(\RR^n)$. We can apply the same argument to any derivative, $D^{\alpha} u$, for all multi-indices $\alpha\in \NN^{n}$, to deduce that $u \in C^{\infty}_0(\RR^n)$, and so the function $u$ belongs to the space $C^{2s+\alpha}_0(\RR^n)$, and $u$ is a solution to equation \eqref{eq:L_0}. Using now the a priori Schauder estimate \eqref{eq:Holder_estimate_prim} together with the supremum estimate \eqref{eq:Sup_estimate}, we obtain that $u$ satisfies inequality \eqref{eq:Holder_estimate}.

We now show that equation \eqref{eq:L_0} has a  solution, $u\in C^{2s+\alpha}(\RR^n)$, for any choice of the source function, $f\in C^{\alpha}(\RR^n)$. We approximate $f$ by a sequence of functions, $\{f_k\}_{k\geq 0}\subset C^{\infty}_c(\RR^n)$, in the sense that
\begin{align*}
&f_k(x)\rightarrow f(x),\quad\hbox{as } k\rightarrow\infty,\quad\forall x\in \RR^n,\\
&\sup_{k\geq 0} \|f_k\|_{C^{\alpha}(\RR^n)} <\infty.
\end{align*}
For each $k \geq 0$, we let $u_k\in C^{2s+\alpha}_0(\RR^n)$ be the unique solution to the equation $L_0 u_k =f_k$ on $\RR^n$. Using estimate \eqref{eq:Holder_estimate}, the Arzel\'a-Ascoli Theorem gives that there is a subsequence, $\{u_k\}_{k\geq 0}$, which for simplicity we denote the same as the initial sequence, which converges to a solution, $u\in C^{2s+\alpha}(\RR^n)$, to equation \eqref{eq:L_0}. The convergence takes place uniformly in $C^{2s+\beta}(K)$, for all $\beta\in (0,\alpha)$, and all compact sets $K\subset \RR^n$.

Because we assume that the coefficient functions $b\in C^{\alpha}(\RR^n;\RR^n)$ and $c\in C^{\alpha}(\RR^n)$, by \cite[Proposition 2.5]{Silvestre_2007}, we see that the operator $L:C^{2s+\alpha} \rightarrow C^{\alpha}(\RR^n)$ is well-defined. Thus, with the aid of the Schauder estimate \eqref{eq:Holder_estimate}, and the existence and uniqueness of solutions in H\"older spaces to the model equation \eqref{eq:L_0}, we can use the method of continuity to prove existence of solutions in H\"older spaces to equation \eqref{eq:Linear_equation}.
\end{proof}

\subsubsection{The penalized equation}
\label{subsec:Penalized_equation}
Before proving existence of solutions to the obstacle problem \eqref{eq:Obstacle_problem} defined by the fractional Laplacian with drift, we first prove existence of solutions to the penalized equation,
\begin{equation}
\label{eq:Penalized_equation}
L u = \beta_{\eps}(\varphi-u)\quad\hbox{on }\RR^n,
\end{equation}
where $\beta_{\eps}:\RR\rightarrow [0,\infty)$ is defined by $\beta_{\eps}(t) = t^+/\eps$, for all $t\in\RR$, and $\eps$ is any positive constant.

\begin{lem}[Existence of solutions to the penalized equation]
\label{lem:Existence_penalized_equation}
Let $s\in(1/2,1)$. Let $\alpha \in (0,1)$ be such that $2s+\alpha$ is not an integer. Assume that the obstacle function $\varphi\in C^{\alpha}(\RR^n)$, and the coefficient functions $b \in C^{\alpha}(\RR^n;\RR^n)$, and $c \in C^{\alpha}(\RR^n)$ satisfies condition \eqref{eq:Nonnegative_lower_bound_c}. Then there is a solution, $u_{\eps}\in C^{2s+\alpha}(\RR^n)$, to the penalized equation \eqref{eq:Penalized_equation}. 
\end{lem}

\begin{proof}
Let $\eps>0$, and consider the operator
$$
L_{\eps} v := Lv + \frac{1}{\eps} v,\quad\forall v \in C^2(\RR^n).
$$
We also let
\begin{equation}
\label{eq:Definition_gamma_eps}
\begin{aligned}
\gamma_{\eps}(v):=\beta_{\eps}(\varphi-v)+\frac{1}{\eps}v
= \frac{1}{\eps}
\begin{cases}
\varphi,&\quad\hbox{if }\varphi>v,\\
v,&\quad\hbox{if }\varphi\leq v.
\end{cases}
\end{aligned}
\end{equation}
We notice that $\gamma_{\eps}$ is a non-decreasing function. We make use of the monotonicity of the nonlinear term $\gamma_{\eps}$, to build a sequence of functions which converges to a solution to the penalized equation \eqref{eq:Penalized_equation}. We let
$$
\underline u = 0,\quad\hbox{and}\quad\overline u = \|\varphi\|_{C(\RR^n)}. 
$$
The functions $\underline u$ and $\overline u$ are chosen such that 
\begin{align*}
\underline u \leq \overline u,\quad L_{\eps} \underline u \leq \gamma_{\eps}(\underline u),\quad\hbox{and}\quad L_{\eps} \overline u \geq \gamma_{\eps}(\overline u).
\end{align*}
We construct iteratively a sequence of functions, $\{u_k\}_{k\geq 0}\subset C^{2s+\alpha}(\RR^n)$, which converge to a solution to the penalized equation \eqref{eq:Penalized_equation}. Let $u_0:=\underline u$, and let $u_k\in C^{2s+\alpha}(\RR^n)$ be the unique solution given by Lemma \ref{lem:Existence_uniqueness_linear_equation} to the linear equation
\begin{equation}
\label{eq:L_eps_equation}
L_{\eps} u_k=\gamma_{\eps}(u_{k-1}),\quad \forall k\geq 1.
\end{equation}
Because we assume that $u_{k-1}\in C^{2s+\alpha}(\RR^n)$ and $\varphi \in C^{\alpha}(\RR^n)$, we see that $\gamma_{\eps}(u_{k-1})\in C^{\alpha}(\RR^n)$, and so, we can apply Lemma \ref{lem:Existence_uniqueness_linear_equation} to build the function $u_k$. Notice that the operator $L_{\eps}$ satisfies the assumptions of Lemma \ref{lem:Existence_uniqueness_linear_equation} because we assume that the coefficient functions  $b\in C^{\alpha}(\RR^n;\RR^n)$, and $c\in C^{\alpha}(\RR^n)$ and satisfies condition \eqref{eq:Nonnegative_lower_bound_c}, and $\eps>0$. We use Lemma \ref{lem:Comparison_principle} to prove inductively that the sequence of solutions, $\{u_k\}_{k\geq 0}$, is non-decreasing, and
\begin{equation}
\label{eq:u_k_increasing}
\underline u=:u_0\leq u_1\leq\ldots\leq u_k \leq \overline u,\quad\forall k \in \NN.
\end{equation}
For $k=1$, we see that the following sequence of inequalities hold
$$
L_{\eps} \underline u \leq \gamma_{\eps}(\underline u) = \gamma_{\eps}(u_0) = L_{\eps} u_1 \leq \gamma_{\eps}(\overline u) \leq L_{\eps} \overline u,
$$
and so, Lemma \ref{lem:Comparison_principle} gives us that the inequality $\underline u\leq u_1 \leq \overline u$ holds. Let now $k\geq 2$, and assume that inequalities \eqref{eq:u_k_increasing} hold with $k$ replaced by $k-1$. Then the monotonicity of $\gamma_{\eps}$ implies that 
\begin{equation*}
\gamma_{\eps}(\underline u)\leq \gamma_{\eps}(u_1)\leq\ldots\leq \gamma_{\eps}(u_{k-1}) \leq \gamma_{\eps}(\overline u),
\end{equation*}
and we have that $L_{\eps} u_{k-1} \leq L_{\eps} u_k \leq L_{\eps} \overline u$. The preceding inequality and Lemma \ref{lem:Comparison_principle} imply that \eqref{eq:u_k_increasing} holds, and the sequence of functions $\{u_k\}_{k\in\NN}$ satisfies
\begin{equation}
\label{eq:Uniform_bound_u_k_eps}
0\leq u_k\leq\|\varphi\|_{C(\RR^n)},\quad\forall k \in\NN.
\end{equation}
From the a priori Schauder estimates \eqref{eq:Holder_estimate}, we have that  
\begin{equation}
\label{eq:Schauder_estimate_u_k_eps}
\|u_k\|_{C^{2s+\alpha}(\RR^n)} \leq C \|\gamma_{\eps}(u_{k-1})\|_{C^{\alpha}(\RR^n)},\quad\forall k \geq 1,
\end{equation}
where $C=C(\alpha, \|b\|_{C^{\alpha}(\RR^n;\RR^n)}, \|c\|_{C^{\alpha}(\RR^n)},\eps,n,s)$ is a positive constant. From definition \eqref{eq:Definition_gamma_eps} of the nonlinear term $\gamma_{\eps}$, we obtain that there is a positive constant, $C=C(\eps)$, such that
\begin{equation}
\label{eq:Holder_norm_gamma_eps}
\|\gamma_{\eps}(u_{k-1})\|_{C^{\alpha}(\RR^n)} \leq C\left(\|u_{k-1}\|_{C^{\alpha}(\RR^n)} + \|\varphi\|_{C^{\alpha}(\RR^n)}\right),\quad\forall k \geq 1,
\end{equation}
and inequalities \eqref{eq:u_k_increasing} and \eqref{eq:Uniform_bound_u_k_eps} give us that
\begin{equation}
\label{eq:Sup_norm_gamma_eps}
\|\gamma_{\eps}(u_{k-1})\|_{C(\RR^n)} \leq C\|\varphi\|_{C(\RR^n)},\quad\forall k \geq 1.
\end{equation}
From \cite[Proposition 2.9]{Silvestre_2007}, we obtain that for any $\beta\in(0,2s-1)$, there is a positive constant, $C=(\beta,n,s)$, such that
$$
\|u_{k-1}\|_{C^{1+\beta}(\RR^n)} \leq C\left(\|(-\Delta)^s u_{k-1}\|_{C(\RR^n)} + \|u_{k-1}\|_{C(\RR^n)}\right),\quad\forall k \geq 1,
$$
which combined with estimates \eqref{eq:Schauder_estimate_u_k_eps} and \eqref{eq:Uniform_bound_u_k_eps}, equation \eqref{eq:L_eps_equation}, and definition of the operator $L_{\eps}$, gives us that
\begin{align*}
\|u_{k-1}\|_{C^{1+\beta}(\RR^n)} &\leq C\left(\|\gamma_{\eps} (u_{k-1})\|_{C(\RR^n)}+\|Du_{k-1}\|_{C(\RR^n)} + \|u_{k-1}\|_{C(\RR^n)}\right),\quad\forall k \geq 1,
\end{align*}
where $C=C(\|b\|_{C(\RR^n;\RR^n)}, \beta, \|c\|_{C(\RR^n)}, c_0, \eps)$ is a positive constant. The Interpolation Inequalities \cite[Theorems 8.8.1]{Krylov_LecturesHolder} yield
\begin{align*}
\|u_{k-1}\|_{C^{1+\beta}(\RR^n)} &\leq C\left(\|\gamma_{\eps} (u_{k-1})\|_{C(\RR^n)}+ \|u_{k-1}\|_{C(\RR^n)}\right),\quad\forall k \geq 1,
\end{align*}
which combined with estimates \eqref{eq:Uniform_bound_u_k_eps}, \eqref{eq:Schauder_estimate_u_k_eps}, \eqref{eq:Holder_norm_gamma_eps} and \eqref{eq:Sup_norm_gamma_eps} give us
$$
\|u_k\|_{C^{2s+\alpha}(\RR^n)} \leq C \|\varphi\|_{C^{\alpha}(\RR^n)},\quad\forall k \in \NN,
$$
if $\alpha\leq 1+\beta$. Because $\beta$ can be chosen in the interval $(0,2s-1)$, and $2s>1$, we  see that the preceding estimate holds for all $\alpha\in (0,1)$. Therefore, the sequence of functions $\{u_k\}_{k\geq 0}$ is uniformly bounded in $C^{2s+\alpha}(\RR^n)$, and we can find a subsequence (which we denote the same as the initial sequence, for simplicity) convergent on compact subsets of $\RR^n$, with respect to the $\|\cdot\|_{C^{2s+\beta}(\RR^n)}$ norm, for all $\beta\in(0,\alpha)$, to a function $u_{\eps} \in C^{2s+\alpha}(\RR^n)$. Moreover, we see that
\begin{align*}
L_{\eps} u_k &\rightarrow Lu_{\eps} + \frac{1}{\eps} u_{\eps},\quad\hbox{as } k \rightarrow \infty,\\
\gamma_{\eps} (u_k) &\rightarrow \beta_{\eps}(\varphi-u_{\eps}) + \frac{1}{\eps} u_{\eps},\quad\hbox{as } k \rightarrow \infty,
\end{align*}
from where is follows that $u_{\eps}$ is a solution to the penalized equation \eqref{eq:Penalized_equation}.
\end{proof}

Before we can apply the existence of solutions to the penalized equation \eqref{eq:Penalized_equation} to prove existence of solutions to the obstacle problem \eqref{eq:Obstacle_problem}, we need the following estimates of the penalization term and the penalization sequence. For each $\eps>0$, let $u_{\eps}\in C^{2s+\alpha}(\RR^n)$ be the solution to equation \eqref{eq:Penalized_equation} constructed in Lemma \ref{lem:Existence_penalized_equation}.

\begin{lem}[Estimates of the penalization term and the penalization sequence]
\label{lem:Estimates_penalization_equation}
In addition to the hypotheses of Lemma \ref{lem:Existence_penalized_equation}, assume that the obstacle function $\varphi\in C_0(\RR^n)$ and obeys condition \eqref{eq:Boundedness_L_phi_positive_part}. Then, for all $\eps>0$, the following estimates hold
\begin{align}
\label{eq:Uniform_bound_beta_eps}
&\|\beta_{\eps}(\varphi-u_{\eps})\|_{C(\RR^n)} \leq \|(L\varphi)^+\|_{C(\RR^n)},\\
\label{eq:Uniform_bound_u_eps}
&0\leq u_{\eps} \leq \|\varphi\|_{C(\RR^n)}.
\end{align} 
\end{lem}

\begin{proof}
Estimate \eqref{eq:Uniform_bound_u_k_eps} gives us \eqref{eq:Uniform_bound_u_eps}. To prove estimate \eqref{eq:Uniform_bound_beta_eps}, we adapt the argument used to prove \cite[Lemma 1.3.1]{Friedman_1982}. Using the fact that  $\varphi\in C_0(\RR^n)$ and that $u_{\eps}$ is a nonnegative function by \eqref{eq:Uniform_bound_u_eps}, we see that the nonlinear term $\beta_{\eps}(\varphi-u_{\eps})$ must attain its global maximum at some point $x_0\in\RR^n$, and $x_0$ is also a global maximum of $\varphi-u_{\eps}$. We obtain that
\begin{align*}
\varphi(x_0)-u_{\eps}(x_0)\geq 0, \quad \nabla \varphi(x_0) - \nabla u_{\eps}(x_0)=0,\quad (-\Delta)^s(\varphi-u_{\eps})(x_0) \geq 0, 
\end{align*}
and using the fact that $c\geq 0$ on $\RR^n$, it follows that $L u_{\eps} (x_0) \leq L\varphi \leq (L\varphi)^+ $. Therefore, $0\leq\beta_{\eps}(\varphi-u_{\eps})(x_0) \leq (L\varphi)^+$, and inequality \eqref{eq:Uniform_bound_beta_eps} follows.
\end{proof}

\subsubsection{Existence of solutions in H\"older spaces to the obstacle problem}
\label{subsec:Existence_solutions_obstacle_problem}
We now use Lemmas \ref{lem:Existence_penalized_equation} and \ref{lem:Estimates_penalization_equation} to prove

\begin{prop}[Existence of solutions in H\"older spaces to the obstacle problem]
\label{prop:Existence_Holder_obstacle_problem}
Let $s\in(1/2,1)$. Let $\alpha\in(0,2s-1)$ be such that $\alpha+2s$ is not an integer. Assume that the obstacle function $\varphi\in C^{\alpha}(\RR^n)\cap C_0(\RR^n)$ is such that condition \eqref{eq:Boundedness_L_phi_positive_part} holds. Assume that the coefficient functions $b \in C^{\alpha}(\RR^n;\RR^n)$, and that $c \in C^{\alpha}(\RR^n)$ satisfies \eqref{eq:Nonnegative_lower_bound_c}. Then there is a solution, $u\in C^{1+\alpha}(\RR^n)$, to the obstacle problem \eqref{eq:Obstacle_problem}, and identity \eqref{eq:Obstacle_problem} holds on $\RR^n$ in the sense of distributions. 
\end{prop}

\begin{rmk}
Notice that Lemma \ref{lem:Existence_penalized_equation} establishes the existence of solutions in $C^{2s+\alpha}(\RR^n)$ to the penalized equation \eqref{eq:Penalized_equation}, while Proposition \ref{prop:Existence_Holder_obstacle_problem} shows the existence of solutions in $C^{1+\alpha}(\RR^n)$ to the obstacle problem \eqref{eq:Obstacle_problem}. When $u\in C^{2s+\alpha}(\RR^n)$, by \cite[Proposition 2.6]{Silvestre_2007} it follows that $(-\Delta u)^s\in C^{\alpha}(\RR^n)$, but when $u \in C^{1+\alpha}(\RR^n)$, we make sense of $(-\Delta)^s u$ only in the distributional sense, as shown in the proof of Proposition \ref{prop:Existence_Holder_obstacle_problem}.
\end{rmk}

\begin{proof}[Proof of Proposition \ref{prop:Existence_Holder_obstacle_problem}]
For each $\eps>0$, let $u_{\eps}\in C^{2s+\alpha}(\RR^n)$ be the solution to the penalized equation \eqref{eq:Penalized_equation} constructed in the proof of Lemma \ref{lem:Existence_penalized_equation}. Using the fact that $\alpha\in(0,2s-1)$, it follows by \cite[Proposition 2.9]{Silvestre_2007} that there is a positive constant, $C=C(\alpha,n,s)$, such that
$$
\|u_{\eps}\|_{C^{1+\alpha}(\RR^n)} \leq C\left(\|u_{\eps}\|_{C(\RR^n)}+\|\beta_{\eps}(\varphi-u_{\eps})\|_{C(\RR^n)}+\|b\dotprod\nabla u_{\eps}\|_{C(\RR^n)} + \|cu_{\eps}\|_{C(\RR^n)}\right).
$$
Using the Interpolation Inequalities \cite[Theorems 8.8.1]{Krylov_LecturesHolder} and Lemma \ref{lem:Estimates_penalization_equation}, we obtain that there is a positive constant, $C=C(\alpha, \|b\|_{C^{\alpha}(\RR^n;\RR^n)}, \|c\|_{C^{\alpha}(\RR^n)}, n, s)$, such that
\begin{equation}
\label{eq:u_eps_Holder_estimate}
\|u_{\eps}\|_{C^{1+\alpha}(\RR^n)} \leq C\left(\|\varphi\|_{C(\RR^n)}+\|(L\varphi)^+\|_{C(\RR^n)}\right),\quad\forall\eps>0.
\end{equation}
Therefore, we can find a subsequence, which for simplicity we denote the same as the initial sequence, which converges locally in $C^{1+\beta}(\RR^n)$, for all $\beta\in(0,\alpha)$, to a function $u \in C^{1+\alpha}(\RR^n)$. 

Notice that because $\alpha\in(0,2s-1)$, the fact that $u \in C^{1+\alpha}(\RR^n)$ does not immediately imply that the quantity $(-\Delta)^s u$ is well-defined. We now make sense of $(-\Delta)^s u$ in the sense of distributions. For this purpose, it is enough to show that $\eta(-\Delta)^s u$ is a tempered distribution, where $\eta:\RR^n\rightarrow [0,1]$ is a smooth cut-off function such that
\begin{equation}
\label{eq:Definition_eta}
\eta \equiv 1\quad\hbox{on }B'_1(x_0),\quad\hbox{and}\quad \eta \equiv 0\quad\hbox{on }(B'_2(x_0))^c,
\end{equation}
and the point $x_0\in\RR^n$ is arbitrarily chosen. We define
\begin{equation}
\label{eq:v_eps}
v:=\eta u,\quad\hbox{and}\quad v_{\eps}:=\eta u_{\eps},\quad\forall \eps>0.
\end{equation}
Direct calculations give us that
\begin{align}
\label{eq:Delta_s_v_cutoff}
(-\Delta)^s v_{\eps} = \eta (-\Delta)^s u_{\eps} + f_{\eps},
\end{align}
where the function $f_{\eps}$ is defined by
\begin{equation}
\label{eq:f_eps}
f_{\eps}(x):=u_{\eps}(x)(-\Delta)^s \eta(x)+\int_{\RR^n}\frac{(\eta(y)-\eta(x))(u_{\eps}(y)-u_{\eps}(x))}{|x-y|^{n+2s}}\ dy,\quad\forall x\in\RR^n.
\end{equation}
Multiplying equation \eqref{eq:Penalized_equation} by $\eta$, we obtain
\begin{align}
\label{eq:Penalized_equation_cutoff}
(-\Delta)^s v_{\eps} -f_{\eps} +\eta(b\dotprod\nabla u_{\eps}+c u_{\eps}) = \eta\beta_{\eps}(\varphi-u_{\eps})\quad\hbox{on }\RR^n.
\end{align}
From definition \eqref{eq:v_eps} of $v_{\eps}$, using the fact that the sequence $\{u_{\eps}\}$ converges locally in $C^{1+\beta}(\RR^n)$ to $u \in C^{1+\alpha}(\RR^n)$, for all $\beta \in (0,\alpha)$, we have the pointwise convergence on $\RR^n$,
\begin{equation}
\label{eq:f_convergence}
-f_{\eps} +\eta(b\dotprod\nabla u_{\eps}+c u_{\eps}) \rightarrow -f + \eta(b\dotprod\nabla u+c u),\quad\hbox{as }\eps\rightarrow 0,
\end{equation}
where $f$ is defined by the same formula as $f_{\eps}$ in \eqref{eq:f_eps}, but with $u_{\eps}$ replaced by $u$. From the definition \eqref{eq:v_eps} of the function $v_{\eps}$, using estimate \eqref{eq:u_eps_Holder_estimate}, and the fact that $v_{\eps}$ and $v$ have compact support contained in $B'_2(x_0)$, we have that
\begin{equation}
\label{eq:v_eps_L_2_convergence}
v_{\eps}\rightarrow v\quad\hbox{in } L^2(\RR^n),\quad\hbox{as } \eps\rightarrow 0.
\end{equation}
We now show that the convergence in  \eqref{eq:v_eps_L_2_convergence} implies the convergence in the sense of distributions,
\begin{equation}
\label{eq:Delta_s_v_eps_distributions_convergence}
(-\Delta)^s v_{\eps} \rightarrow (-\Delta)^s v \quad\hbox{in } \cS'(\RR^n),\quad\hbox{as } \eps\rightarrow 0.
\end{equation}
For all $\psi\in\cS(\RR^n)$, we have that
\begin{align*}
\left\langle (-\Delta)^s v_{\eps}, \psi \right\rangle &= \left\langle \widehat{(-\Delta)^s v_{\eps}}, \widehat \psi \right\rangle\\
&=\left\langle  \widehat v_{\eps}, |\xi|^{2s}\widehat \psi \right\rangle\quad\hbox{(using the fact that $\widehat{(-\Delta)^s v_{\eps}}(\xi)=|\xi|^{2s}\widehat v_{\eps}(\xi)$)},
\end{align*}
where $\left\langle\cdot,\cdot\right\rangle$ denotes the duality of $\cS'(\RR^n)$ and $\cS(\RR^n)$. From \eqref{eq:v_eps_L_2_convergence}, we know that $\widehat v_{\eps}$ converges to $\widehat v$ in $L^2(\RR^n)$, as $\eps\rightarrow 0$, and from the fact that $|\xi|^{2s}\widehat \psi\in L^2(\RR^n)$, for all $\psi\in\cS(\RR^n)$, we obtain that
\begin{align*}
\left\langle (-\Delta)^s v_{\eps}, \psi \right\rangle 
&\rightarrow \left\langle  \widehat v, |\xi|^{2s}\widehat \psi \right\rangle\\
&= \left\langle \widehat{(-\Delta)^s v}, \widehat \psi \right\rangle\\
&= \left\langle (-\Delta)^s v, \psi \right\rangle, 
\end{align*}
from which the convergence in \eqref{eq:Delta_s_v_eps_distributions_convergence} follows. Using \eqref{eq:f_convergence} and \eqref{eq:Delta_s_v_eps_distributions_convergence} , we can let $\eps$ tend to $0$ in \eqref{eq:Penalized_equation_cutoff} to obtain that the following hold in distributional sense,
\begin{equation}
\label{eq:Obstacle_problem_cutoff}
\begin{aligned}
\begin{cases}
(-\Delta)^s (\eta u) - f + \eta(b\dotprod\nabla u+c u) =0,&\quad\hbox{if}\quad\eta\varphi<\eta u,\\
(-\Delta)^s (\eta u) - f + \eta(b\dotprod\nabla u+c u) \geq 0,&\quad\hbox{if}\quad\eta\varphi\geq \eta u,
\end{cases}
\end{aligned}
\end{equation}
where we used the fact that $v=\eta u$. Therefore, on the set $\{\eta=1\}\supseteq B'_1(x_0)$, using \eqref{eq:Delta_s_v_cutoff} applied to $v$ instead of $v_{\eps}$, we obtain
\begin{equation*}
\begin{aligned}
\begin{cases}
(-\Delta)^s  u + b\dotprod\nabla u+c u =0,&\quad\hbox{if}\quad \varphi< u,\\
(-\Delta)^s  u + b\dotprod\nabla u+c u \geq 0,&\quad\hbox{if}\quad\varphi\geq u.
\end{cases}
\end{aligned}
\end{equation*}
Because the point $x_0\in\RR^n$ was arbitrarily chosen, we obtain that identity \eqref{eq:Obstacle_problem} holds on $\RR^n$ in the distributional sense. Therefore, the function $u \in C^{1+\alpha}(\RR^n)$ is a solution to the obstacle problem \eqref{eq:Obstacle_problem}.
\end{proof}

Proposition \ref{prop:Existence_Holder_obstacle_problem} is the main ingredient in the proof of Proposition \ref{prop:Solutions_partial_regularity} together with \cite[Theorem 5.8]{Silvestre_2007}.

\begin{proof}[Proof of Proposition \ref{prop:Solutions_partial_regularity}]
From Proposition \ref{prop:Existence_Holder_obstacle_problem}, we may assume without loss of generality that $\alpha \in [2s-1,1)$.
We improve the regularity of solutions from $u\in C^{1+\gamma}(\RR^n)$, for all $\gamma\in (0,2s-1)$, established in Proposition \ref{prop:Existence_Holder_obstacle_problem}, to  $u\in C^{1+\beta}(\RR^n)$, for all $\beta\in (0,\alpha\wedge s)$, by a bootstrapping argument.

Let $\alpha_0\in (0,2s-1)$, and let $u\in C^{1+\alpha_0}(\RR^n)$ be the solution to the obstacle problem \eqref{eq:Obstacle_problem} constructed in the proof of Proposition \ref{prop:Existence_Holder_obstacle_problem}. Then by the complementarity conditions \eqref{eq:Obstacle_problem_cutoff} and definition \eqref{eq:Definition_eta} of the cut-off function $\eta$, the function $\eta u$ is a solution to the obstacle problem,
$$
\min\{(-\Delta)^s (\eta u) - f + \eta(b\dotprod\nabla u+c u), \eta u-\eta\varphi\}=0\quad\hbox{on } \RR^n,
$$
where the function $f$ is given by the same formula as $f_{\eps}$ in \eqref{eq:f_eps}, but with $u_{\eps}$ replaced by $u$. We define the function $g$ by
\begin{equation}
\label{eq:Definition_g} 
g(x):=f(x)-\eta(x)(b(x)\dotprod\nabla u(x)+c(x) u(x)),\quad\forall x\in \RR^n.
\end{equation}
Because the function $u$ belongs to $C^{1+\alpha_0}(\RR^n)$ and $\eta$ has compact support, then the function $g$ decays like $|x|^{-(n+2s)}$, as $|x|\rightarrow\infty$, and Lemma \ref{lem:Regularity_g} yields that $g$ belongs to $C^{\alpha_0\wedge (2(1-s))}(\RR^n)$. We let $w$ be defined by
\begin{equation}
\label{eq:Definition_w_existence}
w(x):=c_{n,s} \int_{\RR^n}\frac{g(y)}{|x-y|^{n-2s}}\ dy,\quad\forall x\in \RR^n.
\end{equation}
The function $w$ is a solution to the linear equation $(-\Delta)^s w = g$ on $\RR^n$. From \cite[Proposition 2.8]{Silvestre_2007} we have that $w$ belongs to $C^{2s+\alpha_0\wedge (2(1-s))}(\RR^n)$. Using the definition \eqref{eq:Definition_w_existence} of $w$, we also have that $w$ decays like $|x|^{-n}$, as $|x|\rightarrow\infty$. Therefore, we obtain that $\eta u-w$ is a continuous solution to the obstacle problem,
\begin{equation}
\label{eq:Simplified_obstacle_problem}
\begin{aligned}
(-\Delta)^s  (\eta u -w) \geq 0&\quad\hbox{on }\RR^n,\\
(-\Delta)^s (\eta u -w) = 0&\quad\hbox{on } \{u>\varphi\} \cap \{\eta>0\} = \{\eta u-w>\eta\varphi-w\},\\
\eta u-w\geq \eta\varphi-w &\quad\hbox{on }\RR^n,\\
\lim_{|x|\rightarrow\infty} \eta(x)u(x)-w(x)=0.
\end{aligned}
\end{equation}
Because $w\in C^{2s+\alpha_0\wedge (2(1-s))}(\RR^n)$ and $\varphi \in C^{1+\alpha}(\RR^n)$, the obstacle function $\eta \varphi-w$ belongs to $C^{1+\gamma_1}(\RR^n)$, where 
\begin{align*}
1+\gamma_1&=(2s+\alpha_0 \wedge (2(1-s)))\wedge (1+\alpha)\\
&=(2s+\alpha_0)\wedge (1+\alpha).
\end{align*}
The second equality follows from the fact that $\alpha\in (0,1)$ and $2s+2(1-s)=2>1+\alpha$. By \cite[Theorem 5.8]{Silvestre_2007}, we have that $\eta u-w\in C^{1+\beta}(\RR^n)$, for all $\beta<\alpha_1$, where $1+\alpha_1:=(2s+\alpha_0)\wedge (1+\alpha)\wedge (1+s)$. Because the center of the ball $B'_1(x_0)$ in the definition \eqref{eq:Definition_eta} of the cut-off function $\eta$ can be chosen arbitrarily in $\RR^n$, we obtain that $u \in C^{1+\beta}(\RR^n)$, for all $\beta<\alpha_1$. If $2s+\alpha_0 \geq (1+\alpha)\wedge (1+s)$, we obtain that the solution $u$ belongs to $C^{1+\beta}$, for all $\beta\in(0,\alpha\wedge s)$. Otherwise, we repeat the preceding steps, but now we notice that the H\"older exponent $\alpha_0$ can be replaced by $\alpha_0+(2s-1)$, where we recall that the increment $2s-1$ is positive, since we assume that $s\in (1/2,1)$. The fact that $w\in C^{2s+\alpha_0+(2s-1)}(\RR^n)$ gives us that $\eta \varphi-w$ belongs to $C^{1+\gamma_2}(\RR^n)$, where 
$$
1+\gamma_2:=(2s+\alpha_0+(2s-1))\wedge (1+\alpha).
$$
By \cite[Theorem 5.8]{Silvestre_2007}, it follows that $\eta u-w\in C^{1+\beta}(\RR^n)$, for all $\beta<\alpha_2$, where $1+\alpha_2:=(2s+\alpha_0+(2s-1))\wedge (1+\alpha)\wedge (1+s)$, and so, the function $u$ belongs to $C^{1+\beta}(\RR^n)$, for all $\beta<\alpha_2$. We repeat this procedure $k$ times where we choose $k$ such that
\begin{align*}
1+\alpha_k&:=(2s+\alpha_0+(k-1)(2s-1))\wedge (1+\alpha)\wedge (1+s)\\
& = (1+\alpha)\wedge (1+s),
\end{align*}
and the conclusion that the function $u$ belongs to $C^{1+\beta}(\RR^n)$, for all $\beta\in (0,\alpha\wedge s)$, now follows. 
\end{proof}

\subsection{Uniqueness of solutions}
\label{sec:Uniqueness}
We use a probabilistic method to prove uniqueness of solutions to the obstacle problem \eqref{eq:Obstacle_problem} by establishing their stochastic representation. Let $(\Omega, \cF, \PP)$ be a filtered probability space endowed with a filtration, $\{\cF(t)\}_{t\geq 0}$, which satisfies the usual hypotheses of completeness and right-continuity \cite[p. 72]{Applebaum}. Let $N(dt,dy)$ be a Poisson random measure on $[0,\infty)\times(\RR^n\backslash\{O\})$ with L\'evy measure, 
\begin{equation}
\label{eq:Levy_measure}
\nu(dy)=\frac{c_{n,s}}{|y|^{n+2s}}\, dy,
\end{equation}
and let $\widetilde N(dt, dy)=N(dt, dy) - \nu(dy)dt$ denote its compensator. We recall the results on existence and uniqueness of solutions to the stochastic equation,
\begin{equation}
\label{eq:SDE}
dX(t)=-b(X(t))\ dt+\int_{\RR^n\backslash\{O\}} y \widetilde N(dt,dy),\quad t>0,\ X_0=x \in\RR^n.
\end{equation}
If the vector field $b:\RR^n\rightarrow\RR^n$ is a bounded, Lipschitz continuous function, then it follows by \cite[Theorem 6.2.9]{Applebaum}, that there is a unique RCLL (right-continuous with left limit) adapted solution to equation \eqref{eq:SDE}.

\begin{prop}[Uniqueness of solutions]
\label{prop:Uniqueness}
Let $s\in (0,1)$ and $\alpha\in((2s-1)\vee 0,1)$. Assume that the obstacle function $\varphi \in C(\RR^n)$, and that the coefficient function $b\in C(\RR^n;\RR^n)$ is Lipschitz continuous on $\RR^n$, and that the coefficient function $c$ is a Borel measurable function which satisfies condition \eqref{eq:Lower_bound_c}. If $u \in C^{1+\alpha}(\RR^n)$ is a solution to the obstacle problem \eqref{eq:Obstacle_problem}, then $u$ has the stochastic representation
\begin{equation}
\label{eq:Stochastic_representation}
u(x)=\sup_{\tau\in\cT} \EE^x\left[e^{-\int_0^{\tau} c(X(s))\, ds} \varphi(X(\tau))\right],\quad\forall x\in\RR^n,
\end{equation}
where $\cT$ is the set of stopping times with respect to the filtration $\{\cF(t)\}_{t\geq 0}$, and $\{X(t)\}_{t\geq 0}$ is the unique solution to the stochastic differential equation \eqref{eq:SDE}, with initial condition $X(0)=x$.
\end{prop}

\begin{proof} 
Because we assume that $\alpha>2s-1$, we may apply It\^o's lemma \cite[Theorem 4.4.7]{Applebaum} to the function $u \in C^{1+\alpha}(\RR^n)$ and the unique solution, $\{X(t)\}_{t\geq 0}$, to equation \eqref{eq:SDE}, with initial condition $X(0)=x$. We obtain
\begin{align*}
d\left(e^{-\int_0^t c(X(s))\, ds} u(X(t))\right)
&= e^{-\int_0^t c(X(s))\, ds}\left[\left(-c(X(t-)) - b(X(t-))\dotprod\nabla u(X(t-))\right)\, dt\right.\\
&\quad+ \int_{\RR^n\backslash\{O\}}\left(u(X(t-)+y)-u(X(t-))\right)\widetilde N(dt,dy)\\
&\quad+ \left. \int_{\RR^n\backslash\{O\}}\left(u(X(t-)+y)-u(X(t-))-y\dotprod\nabla u(X(t-))\right) \nu(dy) dt\right].
\end{align*}
The assumptions that the function $u$ belongs to $C^{1+\alpha}(\RR^n)$, and $\alpha>2s-1$, is used to ensure that the last term in the preceding expression is well-defined. Because the function $u$ belongs to $C^1(\RR^n)$, and $s>1/2$, we see from definition \eqref{eq:Levy_measure} of the L\'evy measure, $\nu(dy)$, that there is a positive constant, $C$, such that
$$
\int_{\RR^n\backslash\{O\}} \left|u(X(t-)+y)-u(X(t-))\right|^2 \nu(dy) \leq C,\quad\forall t \geq 0,
$$
and so, it follows by the Martingale Representation Theorem \cite[Theorem 5.3.5]{Applebaum} that the process
$$
M(t):=\int_0^t\int_{\RR^n\backslash\{O\}}\left(u(X(s-)+y)-u(X(s-))\right)\widetilde N(ds,dy),\quad t\geq 0,
$$
is a martingale. We then can write
$$
e^{-\int_0^t c(X(s))\, ds} u(X(t))= u(x)- \int_0^t e^{-\int_0^s c(X(r))\, dr} Lu(X(s-))\, ds +M(t),\quad\forall t \geq 0.
$$
As usual, we define the stopping time $\tau^*$ by
$$
\tau^*:=\inf\{t \geq 0:\ u(X(t))=\varphi(X(t))\}.
$$
On the set $\{u>\varphi\}$, we have that $Lu=0$. Moreover, since we assume that the solution $u$ belongs to the H\"older space $C^{1+\alpha}$, for some $\alpha>2s-1$, it follows by \cite[Proposition 2.6]{Silvestre_2007} that the function $Lu$ is continuous on $\RR^n$, and so, we have that $Lu(x)=0$, for all $x\in\{u>\varphi\}\cup\partial\{u>\varphi\}$. We then obtain that the stopped process
$$
\left\{e^{-\int_0^{t\wedge\tau^*} c(X(s))\, ds} u(X(t\wedge\tau^*))\right\}_{t\geq 0}
$$
is a martingale, which gives us that
$$
u(x)=\EE^x\left[e^{-\int_0^{\tau^*} c(X(s))\, ds} \varphi(X(\tau^*))\right],\quad\forall x\in\RR^n,
$$
where $\EE^x$ denotes expectation with respect to the law of the unique solution, $\{X(t)\}_{t\geq 0}$, to the equation \eqref{eq:SDE}, with initial condition $X(0)=x$. The condition that the coefficient function $c$ satisfies inequality \eqref{eq:Lower_bound_c}, is used to ensure that the integrand in the preceding expression is well-defined when $\tau^*=\infty$. Because we assume that the obstacle function $\varphi$ is bounded, when $\tau^*=\infty$, the integrand in the preceding expression is zero. Because $Lu \geq 0$ on $\RR^n$, in general, we have that the process
$$
\left\{e^{-\int_0^{t} c(X(s))\, ds} u(X(t))\right\}_{t\geq 0}
$$
is a supermartingale. Together with the fact that $u\geq \varphi$ on $\RR^n$, this implies that, for all $\tau\in\cT$, we have that
$$
u(x)\geq \EE^x\left[e^{-\int_0^{\tau} c(X(s))\, ds} \varphi(X(\tau))\right].
$$
Thus we obtain the stochastic representation \eqref{eq:Stochastic_representation} of solutions $u$ to the obstacle problem \eqref{eq:Obstacle_problem}, which in particular implies that the solution is unique.
\end{proof}

\subsection{Monotonicity formula}
\label{sec:Monotonicity_formula}
In this section we prove a new Almgren-type monotonicity formula suitable for solutions to the obstacle problem defined by the fractional Laplacian with drift, \eqref{eq:Obstacle_problem}. We use the monotonicity formula to establish the optimal regularity of solutions in \S \ref{sec:Solutions_optimal_regularity}. 

We assume that the hypotheses of Proposition \ref{prop:Solutions_partial_regularity} hold, and in addition that the obstacle function $\varphi$ belongs to $C^{2s+\alpha}(\RR^n)$, for all $\alpha\in (0,s)$. Proposition \ref{prop:Solutions_partial_regularity} gives us the existence of a solution $u\in C^{1+\alpha}(\RR^n)$, for all $\alpha\in (0,s)$, to the obstacle problem \eqref{eq:Obstacle_problem} which solves the ``localized" obstacle problem \eqref{eq:Simplified_obstacle_problem}. We recall that the function $w$ defined in \eqref{eq:Definition_w_existence} and appearing in \eqref{eq:Simplified_obstacle_problem}, belongs to the space $C^{2s+\alpha}(\RR^n)$, and so, the function $\eta \varphi-w$ is contained in $C^{2s+\alpha}(\RR^n)$, for all $\alpha \in (0,s)$, since we assume that $\varphi \in C^{2s+\alpha}(\RR^n)$, for all $\alpha \in (0,s)$. Therefore, using \eqref{eq:Simplified_obstacle_problem} we reduce the study of the regularity of solutions to the obstacle problem \eqref{eq:Obstacle_problem} to that of solutions to the problem,
\begin{equation}
\label{eq:Obstacle_problem_simple}
\begin{aligned}
(-\Delta)^s  u \geq 0&\quad\hbox{on }\RR^n,\\
(-\Delta)^s u = 0&\quad\hbox{on } \{u>\varphi\},\\
 u\geq \varphi &\quad\hbox{on }\RR^n,
\end{aligned}
\end{equation}
where we now let $u$ replace $\eta u-w$, and $\varphi$ replace $\eta\varphi-w$ in \eqref{eq:Simplified_obstacle_problem}. Thus, the natural starting assumption in proving the optimal regularity of solutions to the obstacle problem \eqref{eq:Obstacle_problem}, is that $\varphi\in C^{2s+\alpha}(\RR^n)$, for all $\alpha\in (0,s)$, and $u \in C^{1+\alpha}(\RR^n)$, for all $\alpha\in (0,s)$,  is a solution to problem \eqref{eq:Obstacle_problem_simple}. We recall that the regularity of these solutions is studied in \cite{Caffarelli_Salsa_Silvestre_2008} under the assumption that $\varphi \in C^{2,1}(\RR^n)$. In our case, in general we can only assume that $\varphi\in C^{2s+\alpha}(\RR^n)$, for all $\alpha\in (0,s)$, due to the presence of the lower order terms in the expression of the operator $L$. 

Let $a:=1-2s$. We consider the operator $L_a$ defined, for all $v\in C^2(\RR^n\times\RR_+)$, by
\begin{equation}
\label{eq:L_a}
L_a v(x,y) = \hbox{div }(|y|^{a}\nabla v)(x,y),\quad\forall (x,y)\in \RR^n\times\RR_+.
\end{equation}
The relation between the degenerate-elliptic operator $L_a$ and the fractional Laplacian operator, $(-\Delta)^s$, is investigated in \cite[\S 3]{Caffarelli_Silvestre_2007}, where it is established that $L_a$-harmonic functions, $u$, satisfy
\begin{equation}
\label{eq:Dirichlet_to_Neumann_map}
\lim_{y\downarrow 0} y^a u_y(x,y) = -(-\Delta)^s u(x,0),
\end{equation}
that is, the fractional Laplacian operator, $(-\Delta)^s$, is a Dirichlet-to-Neumann map for the elliptic operator $L_a$. Identity \eqref{eq:Dirichlet_to_Neumann_map} holds up to multiplication by a constant factor (see \cite[Formula (3.1)]{Caffarelli_Silvestre_2007}).

We construct the $L_a$-harmonic extensions of the functions $u(x)$ and $\varphi(x)$ from $\RR^n$ to the half-space $\RR^n\times\RR_+$ (see \cite[\S 2.4]{Caffarelli_Silvestre_2007}). For simplicity, we keep the same notation for the extensions as for the initial functions, even if the domains changed. That is, we denote the extensions of the functions $u(x)$  and $\varphi(x)$, defined for all $x\in \RR^n$, by $u(x,y)$ and $\varphi(x,y)$, defined for all points $(x,y)\in\RR^n\times\RR_+$, respectively. We assume without loss of generality that $O$ is a point on $\partial\{u=\varphi\}$. Given the fact that $\varphi \in C^{2s+\alpha}(\RR^n)$, for all $\alpha \in (0,s)$, we may consider the following auxiliary ``height" function,
\begin{equation}
\label{eq:Auxiliary_function_v}
v(x,y) := u(x,y) - \varphi(x,y) + \frac{1}{2s}(-\Delta)^s\varphi(O)|y|^{1-a},\quad\forall (x,y)\in\RR^n\times\bar\RR_+,
\end{equation}
and we extend $v$ to the whole space $\RR^{n+1}$ by even reflection, i.e. we let $v(x,y)=v(x,-y)$, for all $(x,y)\in\RR^n\times\RR_+$. Compare the definition of the function $v$ with that of $\widetilde u$ in \cite[p. 433]{Caffarelli_Silvestre_2007}, where the condition that the obstacle $\varphi$ belongs to $C^2(\RR^n)$ is required. 

Because we know that the function $u \in C^{1+\alpha}(\RR^n)$ and $\varphi \in C^{2s+\alpha}(\RR^n)$, and $O\in\partial\{u=\varphi\}$, we can find a positive constant, $C$, such that
\begin{equation}
\label{eq:Ineq_v_on_R_n}
0\leq v(x) \leq C|x|^{1+\alpha},\quad\forall x\in\RR^n.
\end{equation}
 In addition, the function $v$ satisfies the following properties
\begin{align}
\label{eq:Properties_v_1}
L_a v =0 &\quad\hbox{on }\RR^n\times(\RR\backslash\{0\}),\\
\label{eq:Properties_v_2}
v \geq 0&\quad\hbox{on }\RR^n\times\{0\}.
\end{align}
The integration by parts formula gives us that
\begin{align*}
L_a v(x,y) &= 2 \lim_{z\downarrow 0} |z|^a v_z(x,z) \cH^n|_{\{y=0\}},
\end{align*}
where $\cH^n|_{\{y=0\}}$ denotes the Hausdorff measure on the hyperplane $\{y=0\}$. Using now identities \eqref{eq:Dirichlet_to_Neumann_map} and \eqref{eq:Auxiliary_function_v}, we obtain
\begin{align}
\label{eq:Equation_v}
L_a v(x,y) &= 2\left(-(-\Delta)^s u(x) +(-\Delta)^s \varphi(x)-(-\Delta)^s \varphi(O)\right)\cH^n|_{\{y=0\}}.
\end{align}
Because the function $u$ solves problem \eqref{eq:Obstacle_problem_simple}, we see that
\begin{align}
\label{eq:Upper_bound_L_a}
L_a v(x,y) &\leq 2\left((-\Delta)^s \varphi(x)-(-\Delta)^s \varphi(O)\right)\cH^n|_{\{y=0\}}\quad\hbox{on } \RR^{n+1},\\
\label{eq:Equality_L_a}
L_a v(x,y) &= 2\left((-\Delta)^s \varphi(x)-(-\Delta)^s \varphi(O)\right)\cH^n|_{\{y=0\}}\quad\hbox{on } \RR^{n+1}\backslash(\{y=0\}\cap\{u=\varphi\}).
\end{align}
Notice that $L_a v$ is a singular measure supported on $\{y=0\}$. Compared to $L_a \widetilde u$, where the function $\widetilde u$ is the analogue of $v$ in \cite[p. 433]{Caffarelli_Silvestre_2007}, the singular measure $L_a v$ has nontrivial support on $\{y=0\}\backslash\{v=0\}$, while the measure $L_a\widetilde u$ is a classical function on $\{y=0\}\backslash\{\widetilde u=0\}$. This is due to the presence of the drift component in the definition of the operator $L$. This difference is one of the key points which makes the analysis of the obstacle problem for the fractional Laplacian with drift \eqref{eq:Obstacle_problem} different that the one of the obstacle problem without drift studied in \cite{Caffarelli_Salsa_Silvestre_2008}.

We denote the right-hand side in inequalities \eqref{eq:Upper_bound_L_a} and \eqref{eq:Equality_L_a} by
\begin{equation}
\label{eq:Definition_h}
h(x):=2\left((-\Delta)^s \varphi(x)-(-\Delta)^s \varphi(O)\right),\quad\forall x\in \RR^n.
\end{equation}
Because $\varphi\in C^{2s+\alpha}(\RR^n)$, we see that $h \in C^{\alpha}(\RR^n)$ by \cite[Proposition 2.6]{Silvestre_2007}, and so we have that
\begin{equation}
\label{eq:Growth_h}
|h(x)| \leq C|x|^{\alpha},\quad \forall x\in \RR^n,
\end{equation}
where $C:=2[(-\Delta)^s \varphi]_{C^{\alpha}(\RR^n)}$, for all $\alpha\in (0,s)$.

Let $U\subseteq\RR^{n+1}$ be a Borel measurable set. We say that a function $w$ belongs to the weighted Sobolev space $H^1(U,|y|^a)$, if $w$ and $Dw$ are function in $L^2_{\loc}(U)$ and
$$
\int_{U}\left(|w|^2+|\nabla w|^2\right)|y|^a <\infty.
$$
From \cite[\S 2.4]{Caffarelli_Silvestre_2007}, it follows that the auxiliary function $v$ belongs to the spaces $C(\RR^{n+1})$ and $ H^1(B_r,|y|^a)$, for all $r>0$. In particular, the following quantities are well-defined:
\begin{align}
\label{eq:F}
F_v(r)&:=\int_{\partial B_r} |v|^2 |y|^a,\\
\label{eq:Phi}
\Phi^p_v(r) &:= r\frac{d}{dr} \log\max\{F_v(r), r^{n+a+2(1+p)}\},
\end{align}
where $r>0$ and $p>0$. The function $F_v(r)$ and $\Phi^p_v(r)$ are the analogues of the functions $F_u(r)$ and $\Phi_u(r)$ given by \cite[Definitions (3.1) and (3.2)]{Caffarelli_Salsa_Silvestre_2008}, but adapted to our framework. 

The main result of this section is the following analogue of \cite[Theorem 3.1]{Caffarelli_Salsa_Silvestre_2008}.

\begin{prop}[Monotonicity formula]
\label{prop:Monotonicity_formula}
Let $s\in (1/2,1)$, $\varphi\in C^{2s+\alpha}(\RR^n)$ and $u\in C^{1+\alpha}(\RR^n)$, for all $\alpha\in (0,s)$, such that $2s+\alpha$ is not an integer. Assume that the function $u$ is a solution to the obstacle problem \eqref{eq:Obstacle_problem_simple}. Then, for all $\alpha\in (2s-1,s)$ and $p\in [s,\alpha+s-1/2)$, there are positive constants, $C$ and $r_0\in (0,1)$, such that the function
\begin{equation}
\label{eq:Monotonicity_formula}
(0,r_0)\ni r\mapsto e^{Cr^{\gamma}} \Phi^p_v(r),
\end{equation}
is non-decreasing, where $\gamma:=2(\alpha+s-p)-1$, and $v$ is defined by identity \eqref{eq:Auxiliary_function_v}.
\end{prop}
The proof of Proposition \ref{prop:Monotonicity_formula} is given in \S \ref{sec:Auxiliary_results_proofs}.

Following \cite[Definition (6.1)]{Caffarelli_Salsa_Silvestre_2008} we introduce the sequence of rescalings, $\{v_r\}_{r >0}$, of the function $v$. For $r\in (0,1)$, we define
\begin{equation}
\label{eq:d_r}
d_r:=\left(\frac{1}{r^{n+a}}\int_{\partial B_r}|v|^2|y|^a\right)^{1/2},
\end{equation}
and we let
\begin{equation}
\label{eq:Rescaling}
v_r(x,y):=\frac{v(r(x,y))}{d_r},\quad\forall (x,y)\in\RR^n\times\RR,
\end{equation}
be a rescaling of the function $v$. With the aid of Proposition \ref{prop:Monotonicity_formula} we prove the following analogue of \cite[Lemma 6.1]{Caffarelli_Salsa_Silvestre_2008}.
\begin{prop}
\label{prop:Phi_at_0}
Suppose that the assumptions of Proposition \ref{prop:Monotonicity_formula} are satisfied. Then, for all $p\in [s,2s-1/2)$, the following hold.
If
\begin{equation}
\label{eq:Fraction_d_r_r_power_finite}
\liminf_{r\downarrow 0}\frac{d_r}{r^{1+p}} <\infty,
\end{equation}
then
\begin{equation}
\label{eq:Phi_at_0_p}
\Phi^p_v(0+) = n+a+2(1+p),
\end{equation}
and if
\begin{equation}
\label{eq:Fraction_d_r_r_power_infty}
\liminf_{r\downarrow 0}\frac{d_r}{r^{1+p}} =\infty,
\end{equation}
then
\begin{equation}
\label{eq:Phi_at_0}
\Phi^p_v(0+) \geq n+a+2(1+s).
\end{equation}
\end{prop}

Proposition \ref{prop:Phi_at_0} shows that the smallest value that the function $\Phi^p_v(r)$ can take is $n+a+2(1+s)$. This property is crucial in the proof of the optimal regularity of the solutions to the obstacle problem in \S \ref{sec:Solutions_optimal_regularity}.

The proof of Proposition \ref{prop:Phi_at_0} relies on the fact that the sequence of rescalings, $\{v_r\}_{r \geq 0}$, contains a subsequence strongly convergent in $H^1(B_1,|y|^a)$, as $r$ tends to $0$. To obtain this, we first prove a series of preliminary results. In Lemma \ref{lem:Uniform_boundedness_H_1_rescalings}, we prove the uniform boundedness in $H^1(B_1,|y|^a)$ of the sequence of rescalings, which is then used in Lemma \ref{lem:Uniform_boundedness_rescalings} to show the uniform boundedness in $L^{\infty}(B_{1/2})$, employing the Moser iterations technique. Lemma \ref{lem:Uniform_boundedness_H_1_rescalings} is not sufficient to conclude the strong convergence in $H^1(B_1,|y|^a)$ of the sequence of rescalings, as $r$ tends to $0$, and so, in Lemmas \ref{lem:Uniform_boundedness_C_1_alpha_n_ball_rescalings} and \ref{lem:Uniform_boundedness_C_1_alpha_n_plus_1_ball_rescalings} we improve the control we have on the sequence of rescalings by proving a uniform bound in H\"older spaces. The results of Lemmas \ref{lem:Uniform_boundedness_C_1_alpha_n_ball_rescalings} and \ref{lem:Uniform_boundedness_C_1_alpha_n_plus_1_ball_rescalings} have their analogues in \cite[Lemma 4.1 and Proposition 4.3]{Caffarelli_Salsa_Silvestre_2008}, respectively. The proofs of the latter results in \cite{Caffarelli_Salsa_Silvestre_2008} rely on the properties of the function $\widetilde u$, defined on \cite[p. 433]{Caffarelli_Salsa_Silvestre_2008}. The analogue in our case of the function $\widetilde u$ in \cite{Caffarelli_Salsa_Silvestre_2008} is the function $v$ defined in \eqref{eq:Auxiliary_function_v}. The function $v$ does not satisfy the properties  of function $\widetilde u$, because $L_a v$ is a singular measure with nontrivial support on $\{y=0\}\backslash\{v=0\}$. For this reason, we cannot adapt the approach of \cite{Caffarelli_Salsa_Silvestre_2008} to our framework, and so we proceed in a different way which we outline in the sequel.

We begin with
\begin{lem}[Uniform boundedness in $H^1(B_1,|y|^a)$]
\label{lem:Uniform_boundedness_H_1_rescalings}
We assume that the hypotheses of Proposition \ref{prop:Monotonicity_formula} hold. Let $\alpha\in (1/2,s)$, and $p \in [s,\alpha+s-1/2)$, and assume that condition \eqref{eq:Fraction_d_r_r_power_infty} holds. Then there are positive constants, $C$ and $r_0$, such that
\begin{equation}
\label{eq:Uniform_boundedness_H_1_rescalings}
\|v_r\|_{H^{1}(B_1,|y|^a)} \leq C,\quad\forall r\in (0,r_0).
\end{equation}
\end{lem} 

\begin{proof}
From identity \eqref{eq:Rescaling}, the following hold, for all $r>0$,
\begin{align}
\label{eq:Formula_grad_v_r}
\int_{B_1}|\nabla v_r|^2 |y|^a &=\frac{r\int_{B_r} |\nabla v|^2|y|^a}{\int_{\partial B_r} |v|^2|y|^a},\\
\label{eq:v_r_partial_B_1}
\int_{\partial B_1} |v_r|^2 |y|^a&=1.
\end{align}
From \eqref{eq:d_r} and condition \eqref{eq:Fraction_d_r_r_power_infty}, there is a positive constant, $r_0$, such that
\begin{equation}
\label{eq:Inequality_d_r}
\int_{\partial B_r}|v|^2|y|^a \geq r^{n+a+2(1+p)},\quad\forall r\in (0,r_0),
\end{equation}
and so, identities \eqref{eq:F} and \eqref{eq:Phi} give us that
$$
\Phi^p_v(r) = r\frac{d}{dr} \log \int_{\partial B_r}|v|^2|y|^a, \quad\forall r\in (0,r_0).
$$
From identities \eqref{eq:Identity_Phi_p} and \eqref{eq:v_v_nu}, it follows that
\begin{align}
\Phi^p_v(r) &= r\frac{2\int_{B_r}|\nabla v|^2|y|^a + 2\int_{B_r} v L_a v}{\int_{\partial B_r}|v|^2|y|^a}+n+a\notag\\
\label{eq:Expansion_Phi}
&= 2\int_{B_1}|\nabla v_r|^2 |y|^a + 2 \frac{r\int_{B_r} v L_a v}{\int_{\partial B_r} |v|^2|y|^a}+n+a\quad\hbox{(by identity \eqref{eq:Formula_grad_v_r}).}
\end{align}
Using \eqref{eq:Formula_grad_v_r}, \eqref{eq:Ineq_v_on_R_n}, \eqref{eq:Auxiliary_function_v}, \eqref{eq:Equality_L_a}, \eqref{eq:Definition_h} and \eqref{eq:Growth_h}, together with the preceding inequality, we see that
\begin{equation}
\label{eq:Inequality_second_term_Phi}
\frac{\left|r\int_{B_r} v L_a v\right|}{\int_{\partial B_r} |v|^2|y|^a} \leq \frac{Cr^{1+(1+\alpha)+\alpha+n}}{r^{2(1+p)+n+a}} = Cr^{2(\alpha-p+s-1/2)},\quad\forall r\in (0,r_0).
\end{equation}
From our assumption that $\alpha\in (1/2,s)$ and $p \in [s,\alpha+s-1/2)$, the right-hand side in the preceding inequality tends to zero, as $r\rightarrow 0$. By Proposition \ref{prop:Monotonicity_formula} and identity \eqref{eq:Expansion_Phi}, we obtain that there are positive constants, $C$ and $r_0$, such that
\begin{equation}
\label{eq:Uniform_bound_grad_v_r}
\int_{B_1}|\nabla v_r|^2 |y|^a  \leq C,\quad\forall r\in (0,r_0).
\end{equation}
By \cite[Lemma 2.12]{Caffarelli_Salsa_Silvestre_2008}, we obtain that, for some positive constant $C=C(n,s)$, we have that
$$
\int_{\partial B_1} |v_r(x,y)-v_r(t(x,y))|^2|y|^a \leq C(1-t) \int_{B_1} |\nabla v_r|^2|y|^a,\quad\forall t \in (0,1).
$$
(Notice that on the right-hand side of the Poincar\'e inequality in \cite[Lemma 2.12]{Caffarelli_Salsa_Silvestre_2008}, the factor $(1-t)$ is missing.) Because the uniform bound \eqref{eq:v_r_partial_B_1} holds, the preceding inequality gives us that
$$
\int_{\partial B_1} |v_r(t(x,y))|^2|y|^a \leq 2C(1-t) \int_{B_1} |\nabla v_r|^2|y|^a+2,\quad\forall t \in (0,1),
$$
and multiplying by $t^a$, and integrating in the $t$-variable, we obtain
\begin{align}
\int_{B_1} |v_r|^2|y|^a &=\int_0^1\int_{\partial B_1} |v_r(t(x,y))|^2|ty|^a\notag\\
&\leq 2C \int_{B_1} |\nabla v_r|^2|y|^a\int_0^1(1-t)t^a\, dt+2\int_0^1t^a\, dt\notag\\
\label{eq:Uniform_bound_v_r}
&\leq C, \quad \forall r \in (0,r_0),
\end{align}
where $C$ is a positive constant. The last inequality follows from the uniform bound \eqref{eq:Uniform_bound_grad_v_r} and the fact that the constant $a\in (-1,0)$, since we assume that $s\in (1/2,1)$. 

Inequalities \eqref{eq:Uniform_bound_v_r} and \eqref{eq:Uniform_bound_grad_v_r} give us \eqref{eq:Uniform_boundedness_H_1_rescalings}. This completes the proof.
\end{proof}

As a consequence of Lemma \ref{lem:Uniform_boundedness_H_1_rescalings} we have

\begin{lem}[Uniform boundedness in $L^{\infty}(B_{1/2})$]
\label{lem:Uniform_boundedness_rescalings}
Suppose that the assumptions of Lemma \ref{lem:Uniform_boundedness_H_1_rescalings} hold. Then there are positive constants, $C$ and $r_0$, such that
\begin{equation}
\label{eq:Uniform_boundedness_rescalings}
\|v_r\|_{L^{\infty}(B_{1/2})} \leq C,
\quad\forall r\in (0,r_0).
\end{equation}
\end{lem} 


\begin{proof}
We prove the supremum estimate \eqref{eq:Uniform_boundedness_rescalings} using the Moser iterations method. We let $\eta:\RR^{n+1}\rightarrow [0,1]$ be a smooth function with compact support in $B_1$. For $r>0$, we let
\begin{equation}
\label{eq:Choice_k}
k:=\frac{r^{2s+\alpha}}{d_r},
\end{equation}
and we consider the following auxiliary functions,
\begin{align*}
q:=v_r^{\pm}+k,\quad\hbox{and}\quad w:=\eta^2(q^{\beta}-k^{\beta}),
\end{align*}
where $\beta$ is a positive constant. From \eqref{eq:Auxiliary_function_v}, \eqref{eq:Equality_L_a} \eqref{eq:Definition_h} and \eqref{eq:Rescaling}, we have that
\begin{align*}  
L_a v_r &=0\quad\hbox{on } B_1\backslash B_1',\\
L_a v_r &=\frac{r^{1-a}}{d_r} h(rx) \cH^{n}|_{\{y=0\}}\quad\hbox{on } B_1\backslash \{y=0, v_r=0\}.
\end{align*}
Because $w=0$ when $v_r=0$, the preceding identities give us that
\begin{align}
\label{eq:Operator_applied_to_auxiliary_function_1}
\int_{B_1} w L_a v_r &= \frac{r^{1-a}}{d_r} \int_{B_1'} h(rx)w(x).
\end{align}
Recall that we also have that
\begin{align}
\label{eq:Operator_applied_to_auxiliary_function_2}
\int_{B_1} w L_a v_r &=-\int_{B_1}\nabla v_r \dotprod\nabla w|y|^a.
\end{align}
Using the definitions of the functions $p$ and $w$, we obtain the identities
\begin{align*}
\nabla v_r \dotprod\nabla w &= \pm \nabla q\dotprod\left(\beta\eta^2q^{\beta-1}\nabla q+2\eta\nabla\eta \left(q^{\beta}-k^{\beta}\right)\right),\\
q^{\beta-1} |\nabla q|^2 &= \frac{4}{(\beta+1)^2} \left|\nabla q^{(\beta+1)/2}\right|^2,
\end{align*}
which combined with identities \eqref{eq:Operator_applied_to_auxiliary_function_1} and \eqref{eq:Operator_applied_to_auxiliary_function_2}, and the fact that $0\leq w \leq \eta^2 q^{\beta}$, gives us that 
\begin{align}
\label{eq:Operator_applied_to_auxiliary_function_int_by_parts}
\frac{4\beta}{(\beta+1)^2} \int_{\RR^{n+1}} \left|\nabla q^{(\beta+1)/2}\right|^2 \eta^2 |y|^a
&\leq \int_{\RR^{n+1}} 2\eta|\nabla\eta||\nabla q| q^{\beta} |y|^a + \frac{r^{1-a}}{d_r} \int_{B_1'} |h(rx)| q^{\beta} \eta^2.
\end{align}
We also have that
\begin{align*}
\int_{\RR^{n+1}} 2\eta|\nabla\eta||\nabla q| q^{\beta} |y|^a 
&=\frac{4}{\beta+1} \int_{\RR^{n+1}}\eta|\nabla\eta| q^{(\beta+1)/2} \left|\nabla q^{(\beta+1)/2}\right||y|^a\\
&\leq \frac{4\beta\eps}{(\beta+1)^2} \int_{\RR^{n+1}}\eta^2 \left|\nabla q^{(\beta+1)/2}\right|^2|y|^a
+\frac{1}{4\eps\beta} \int_{\RR^{n+1}}|\nabla\eta|^2 q^{\beta+1} |y|^a,
\end{align*}
for all $\eps>0$. We choose $\eps=1/2$ in the preceding inequality, which we combine with inequalities \eqref{eq:Growth_h} and \eqref{eq:Operator_applied_to_auxiliary_function_int_by_parts}, the fact that $q \geq k$ and $a=1-2s$, to obtain that there is a positive constant, $C$, such that
\begin{align}
\label{eq:Estimate_grad_p_1}
\int_{\RR^{n+1}} \left|\nabla q^{(\beta+1)/2}\right|^2 \eta^2 |y|^a \leq C\int_{\RR^{n+1}} |\nabla\eta|^2 q^{\beta+1} |y|^a
+C\beta\frac{r^{2s+\alpha}}{d_r} \int_{B_1'} \frac{q^{\beta+1}}{k}\eta^2.
\end{align}
By the Trace Theorem \cite[Theorem 1.5.1.1]{Grisvard} applied with $s=1/2+\sigma$, and $\sigma \in (0,1/2)$, we obtain
\begin{equation*}
\int_{\RR^n} \left|q^{(\beta+1)/2}\eta\right|^2 \leq C\|q^{(\beta+1)/2}\eta\|^2_{H^{1/2+\sigma}(\RR^{n+1})}.
\end{equation*}
The Interpolation Inequality \cite[Theorem 1.4.3.3]{Grisvard} gives us, by choosing $\sigma:=1/4$ in the preceding estimate, that there is a positive constant, $C$, such that for all $\eps\in (0,1)$, we have 
\begin{align*}
\|q^{(\beta+1)/2}\eta\|^2_{H^{1/2+\sigma}(\RR^{n+1})} 
&\leq \eps\|q^{(\beta+1)/2}\eta\|^2_{H^1(\RR^{n+1})} + C\eps^{-3} \|q^{(\beta+1)/2}\eta\|^2_{L^2(\RR^{n+1})}\\
&\leq \eps\int_{B_1} \eta^2 \left|\nabla q^{(\beta+1)/2}\right|^2 + C\eps^{-3}\int_{B_1} \left(\eta^2+|\nabla\eta|^2\right)q^{\beta+1}.
\end{align*}
Combining the previous two inequalities, we obtain
\begin{equation}
\label{eq:Estimate_p_beta_plus_1}
\int_{\RR^n} \eta^2 q^{\beta+1} \leq \eps\int_{B_1} \eta^2 \left|\nabla q^{(\beta+1)/2}\right|^2 + C\eps^{-3}\int_{B_1} \left(\eta^2+|\nabla\eta|^2\right)q^{\beta+1},\quad\forall \eps>0.
\end{equation}
We choose
$$
\eps:=\frac{k d_r}{2C\beta r^{2s+\alpha}}.
$$
Using inequalities \eqref{eq:Estimate_p_beta_plus_1} and \eqref{eq:Estimate_grad_p_1}, together with the fact that $a=1-2s<0$, when $s\in (1/2,1)$, we obtain
\begin{align*}
\int_{\RR^{n+a}} \eta^2 \left|\nabla q^{(\beta+1)/2}\right|^2  |y|^a 
\leq \left(\frac{C\beta r^{2s+\alpha}}{k d_r}\right)^3\int_{\RR^{n+1}} \left(\eta^2+|\nabla\eta|^2\right) q^{\beta+1} |y|^a.
\end{align*}
The choice \eqref{eq:Choice_k} of the constant $k$ now gives us that
\begin{align*}
\int_{\RR^{n+1}} \eta^2 \left|\nabla q^{(\beta+1)/2}\right|^2  |y|^a \leq C \beta^3\int_{\RR^{n+1}} \left(\eta^2+|\nabla\eta|^2\right) q^{\beta+1} |y|^a.
\end{align*}
We can now apply the Moser iteration method to conclude that
\begin{equation}
\label{eq:Sup_estimate_p}
\sup_{B_{1/2}} q \leq C \left(\int_{B_1} |q|^2|y|^a\right)^{1/2}.
\end{equation}
The Moser iteration method is applied as in \cite[p. 195-197]{GilbargTrudinger} with the observation that we replace the application of the classical Sobolev inequality \cite[Inequality (7.26)]{GilbargTrudinger} with the Sobolev inequality suitable for the weighted Sobolev space $H^1(B_1,|y|^a)$ obtained in \cite[Theorem (1.6)]{Fabes_Kenig_Serapioni_1982a}. We apply \cite[Theorem (1.6)]{Fabes_Kenig_Serapioni_1982a} to the function $\eta q^{(\beta+1)/2}$ with $p=2$, and we notice that the weight $\fw(x,y)=|y|^a$ belongs to the Muckenhoupt $A_2$ class of functions.

From definition \eqref{eq:Choice_k} of the constant $k$, definition \eqref{eq:d_r} of $d_r$ and inequality \eqref{eq:Inequality_d_r}, it follows that
$$
k \leq r^{2s+\alpha-(1+p)},
$$
and so, using the definition of the auxiliary function $q$, estimate \eqref{eq:Sup_estimate_p} gives that
$$
\sup_{B_{1/2}} v_r^{\pm} \leq C \left(\int_{B_1} |v_r|^2|y|^a\right)^{1/2} + Cr^{2s+\alpha-(1+p)}.
$$
From our assumption that $s\in (1/2, 1)$, $\alpha\in (1/2,s)$ and $p \in (s,\alpha+s-1/2)$, we see that the term $r^{2s+\alpha-(1+p)}$ tends to zero, as $r\rightarrow 0$. From estimate \eqref{eq:Uniform_boundedness_H_1_rescalings}, it follows that
$$
\sup_{B_{1/2}} v_r^{\pm} \leq C,
$$
for some positive constant $C$, and for all $r>0$ sufficiently small. The preceding estimate is equivalent to \eqref{eq:Uniform_boundedness_rescalings}.
\end{proof}

\begin{lem}[Uniform Schauder estimates on $B'_{1/4}$]
\label{lem:Uniform_boundedness_C_1_alpha_n_ball_rescalings}
Let $s\in (1/2,1)$ and $\alpha\in ((2s-1)\vee 1/2, s)$. Suppose that $\varphi\in C^{2s+\alpha}(\RR^n)$, $u\in C^{1+\alpha}(\RR^n)$, and that $u$ is a solution to problem \eqref{eq:Obstacle_problem_simple}. Let $p \in (s,\alpha+s-1/2)$, and assume that condition \eqref{eq:Fraction_d_r_r_power_infty} holds. Then for all $\beta \in (0,2s-1)$, there are positive constants, $C$ and $r_0$, such that
\begin{equation}
\label{eq:Uniform_boundedness_C_1_alpha_n_ball_rescalings}
\|v_r\|_{C^{1+\beta}(B'_{1/4})} \leq C,\quad\forall r\in (0,r_0).
\end{equation}
\end{lem} 

\begin{proof}
In this proof we do not use the $L_a$-harmonic extension of the rescaling sequence, $\{v_r\}_{r>0}$. We divide the proof into several steps. In Step \ref{step:Uniform_bound_h_r_w_r}, we replace the sequence of rescalings, $\{v_r\}_{r>0}$, by a suitable modification \eqref{eq:Definition_v_r_modified} which solves the obstacle problem \eqref{eq:Obstacle_problem_w_r}, where now a non-zero source function, $h_r$, appears. We prove that the sequence of source functions, $\{h_r\}_{r>0}$, satisfies the uniform supremum estimate \eqref{eq:Uniform_bound_h_r_w_r}. In Step \ref{step:Localization_w_r}, we localize our sequence of modified rescaling functions, and we use inequality \eqref{eq:Uniform_bound_h_r_w_r} to prove the uniform global Schauder estimate \eqref{eq:Schauder_estimate_w_r_cutoff}. Finally, in Step \ref{step:Holder_continuity}, we use a localization method of Krylov \cite[Theorem 8.11.1]{Krylov_LecturesHolder} to prove the uniform Schauder estimate \eqref{eq:Uniform_boundedness_C_1_alpha_n_ball_rescalings} satisfied by the sequence of rescalings, $\{v_r\}_{r>0}$.

\begin{step}
\label{step:Uniform_bound_h_r_w_r}
We recall from \eqref{eq:Auxiliary_function_v} and \eqref{eq:Rescaling} that, restricted to the hyperplane $\{y=0\}$, the rescaling functions, $v_r$, take the following form
\begin{equation}
\label{eq:v_r_on_R_n}
v_r(x)=\frac{u(rx)-\varphi(rx)}{d_r},\quad\forall x\in\RR^n,\ \forall r>0,
\end{equation}
where $d_r$ is defined in \eqref{eq:d_r}. Because $u$ is a solution to the obstacle problem
\begin{equation}
\label{eq:Obstacle_problem_u}
\min\{(-\Delta)^s u, u-\varphi\}=0\quad\hbox{on }\RR^n,
\end{equation}
we see that $v_r$ is a solution to the obstacle problem
\begin{equation}
\label{eq:Obstacle_problem_v}
\min\{(-\Delta)^s v_r- f_r, v_r\}=0\quad\hbox{on }\RR^n,
\end{equation}
where the function $f_r$ is defined by
\begin{equation*}
f_r(x):=-\frac{r^{2s}}{d_r}(-\Delta)^s\varphi(rx),\quad\forall x\in\RR^n.
\end{equation*}
We would like to derive a uniform bound on $\|f_r\|_{L^{\infty}(B'_1)}$, for all $r>0$ sufficiently small. Because we do not have a uniform estimate on $r^{2s}/d_r(-\Delta)^s\varphi(O)$, for $r>0$ sufficiently small, we are not able to find an uniform bound on $\|f_r\|_{L^{\infty}(B'_1)}$, and so we choose a different approach. 
We replace the rescaling functions \eqref{eq:v_r_on_R_n} with the following modified version,
\begin{equation}
\label{eq:Definition_v_r_modified}
w_r(x):=\frac{u(rx)-\varphi(rx)+\psi(rx)}{d_r},\quad\forall x\in\RR^n,\quad \forall r>0.
\end{equation}
We define the auxiliary function $\psi$ by
\begin{equation}
\label{eq:Definition_psi_w_r}
\psi(x):=c|x|^4\eta(x),\quad\forall x\in\RR^n,\ \forall r>0,
\end{equation}
where the positive constant $c$ will be suitably chosen below, and $\eta:\RR^n\rightarrow [0,1]$ is a smooth cut-off function with support in $B'_1$. Because $u$ solves the obstacle problem \eqref{eq:Obstacle_problem_u}, we see that the function $w_r$ solves the problem 
\begin{equation}
\label{eq:Obstacle_problem_w_r}
\min\left\{(-\Delta)^s w_r(x) - h_r(x), w_r(x)-\frac{\psi(rx)}{d_r}\right\}=0,\quad\forall x\in \RR^n.
\end{equation}
where the source function $h_r$ is now given by
\begin{equation}
\label{eq:Definition_h_r_w_r}
h_r(x):=\frac{r^{2s}}{d_r}\left((-\Delta)^s\psi(rx)-(-\Delta)^s\varphi(rx)\right),\quad\forall x\in\RR^n.
\end{equation}
Our goal in this step is to show that there are positive constants, $C$ and $r_0$, such that
\begin{equation}
\label{eq:Uniform_bound_h_r_w_r}
\|h_r\|_{L^{\infty}(B'_1)} \leq C,\quad\forall r\in (0,r_0).
\end{equation}
We can rewrite $h_r$ in the form
\begin{align*}
h_r(x)
&= \frac{r^{2s}}{d_r} \left((-\Delta)^s\psi(rx) - (-\Delta)^s\psi(O)\right) 
- \frac{r^{2s}}{d_r} \left((-\Delta)^s\varphi(rx) - (-\Delta)^s\varphi(O)\right)\\
&\quad+\frac{r^{2s}}{d_r} \left((-\Delta)^s\psi(O)) - (-\Delta)^s\varphi(O)\right).
\end{align*}
We choose the constant $c$, in definition \eqref{eq:Definition_psi_w_r} of $\psi$, such that
$$
(-\Delta)^s\psi(O) = (-\Delta)^s\varphi(O).
$$
Because $\varphi \in C^{2s+\alpha}(\RR^n)$ and $\psi\in C^{\infty}_c(\RR^n)$, we obtain by \cite[Proposition 2.6]{Silvestre_2007} that there is a positive constant, $C$, such that
\begin{align*}
\left|(-\Delta)^s\varphi(rx) - (-\Delta)^s\varphi(O)\right|&\leq C r^{\alpha},\\
\left|(-\Delta)^s\psi(rx) - (-\Delta)^s\psi(O)\right|&\leq C r^{\alpha},
\end{align*}
for all $x \in B'_1$ and all $r>0$. Thus, we obtain that 
\begin{equation}
\label{eq:Estimate_h_r_w_r}
|h_r(x)| \leq C\frac{r^{2s+\alpha}}{d_r},\quad\forall x\in B'_1,\quad \forall r>0.
\end{equation}
From definition \eqref{eq:d_r} of $d_r$, and inequality \eqref{eq:Inequality_d_r} (implied by condition \eqref{eq:Fraction_d_r_r_power_infty}), it follows that
$$
\frac{r^{2s+\alpha}}{d_r} \leq r^{2s+\alpha-(1+p)},\quad\forall r\in (0,r_0),
$$
and, from our assumption that $s\in (1/2,1)$, $\alpha\in (1/2,s)$ and $p \in (s,\alpha+s-1/2)$, we have that the bound $r^{2s+\alpha-(1+p)}$ tends to zero, as $r\rightarrow 0$. We can now see that \eqref{eq:Uniform_bound_h_r_w_r} follows from inequality \eqref{eq:Estimate_h_r_w_r} and the preceding observation.
\end{step}

\begin{step}[Localization]
\label{step:Localization_w_r}
In Step \ref{step:Uniform_bound_h_r_w_r}, we obtained the uniform estimate \eqref{eq:Uniform_bound_h_r_w_r} on $B'_1$, but not on $\RR^n$. To be able to use this estimate in Step \ref{step:Holder_continuity}, we need to localize the sequence of rescalings, $\{w_r\}_{r>0}$, defined in \eqref{eq:Definition_v_r_modified}. We do this by simply multiplying the function $w_r$ by a suitably chosen smooth cut-off function, $\chi:\RR^n\rightarrow [0,1]$, with compact support in $B'_1$. We denote
\begin{equation}
\label{eq:Definition_w_r_cutoff}
w^{\chi}_r:=\chi w_r ,\quad\forall r>0.
\end{equation}
We next want to show that, for all $\beta\in (0,2s-1)$, there is a positive constant, $C=C(\beta)$, such that the following estimate holds, for all $r>0$,
\begin{equation}
\label{eq:Schauder_estimate_w_r_cutoff}
\|w^{\chi}_r\|_{C^{1+\beta}(\RR^n)} \leq C\left(\|g_r\|_{C(\RR^n)}+\left\|(-\Delta)^s\left(\frac{\psi(rx)}{d_r}\chi\right)\right\|_{C(\RR^n)}+\left\|\frac{\psi(rx)}{d_r}\chi\right\|_{C(\RR^n)}\right).
\end{equation}
Direct calculations give us
\begin{align*}
(-\Delta)^s w^{\chi}_r(x) = \chi(x)(-\Delta)^s w_r(x) + w_r(x) (-\Delta)^s\chi(x)-\int_{\RR^n}\frac{(\chi(x)-\chi(y))(w_r(x)-w_r(y))}{|x-y|^{n+2s}}\, dy.
\end{align*}
We let the function $g_r$ be defined by
\begin{align}
\label{eq:Definition_g_r}
g_r(x):= \chi(x)h_r(x)+ w_r(x) (-\Delta)^s\chi(x)+ w^{\chi}_r(x)-\int_{\RR^n}\frac{(\chi(x)-\chi(y))(w_r(x)-w_r(y))}{|x-y|^{n+2s}}\, dy,
\end{align}
where we recall that the function $h_r$ is defined in \eqref{eq:Definition_h_r_w_r}. Because $w_r$ solves the obstacle problem \eqref{eq:Obstacle_problem_w_r}, we see that $w^{\chi}_r$ solves the problem
\begin{equation*}
\min\left\{(-\Delta)^s w^{\chi}_r(x) +w^{\chi}_r(x) - g_r(x), w^{\chi}_r(x)-\frac{\psi(rx)}{d_r}\chi(x)\right\}=0,\quad\forall x\in \RR^n.
\end{equation*}
Because $\psi \in C^{\infty}_c(\RR^n)$ and $\varphi\in C^{2s+\alpha}(\RR^n)$, it follows from definition \eqref{eq:Definition_h_r_w_r} of $h_r$, and from \cite[Proposition 2.6]{Silvestre_2007} that $h_r$ belongs to $C^{\alpha}(\RR^n)$. From Lemma \ref{lem:Regularity_g}, we obtain that the function
$$
\RR^n\ni x\mapsto  \int_{\RR^n}\frac{(\chi(x)-\chi(y))(w_r(x)-w_r(y))}{|x-y|^{n+2s}}\, dy
$$
belongs to $C^{2(1-s)}(\RR^n)$. Thus, using definition \eqref{eq:Definition_g_r} of the function $g_r$, and the fact that $u$ belongs to $C^{1+\alpha}(\RR^n)$, we obtain that $g_r \in C^{\theta}(\RR^n)$, where $\theta:=\alpha\wedge (2(1-s))$. We may now apply Lemma \ref{lem:Existence_uniqueness_linear_equation} to conclude that there is a unique solution, $t_r\in C^{2s+\theta}(\RR^n)$, to the linear equation
$$
(-\Delta)^s t_r(x) +t_r(x) = g_r(x),\quad\forall x\in \RR^n.
$$
Then the function $w^{\chi}_r-t_r$ solves the obstacle problem
\begin{equation}
\label{eq:Obstacle_problem_w_r_cutoff}
\min\left\{(-\Delta)^s (w^{\chi}_r-t_r)+(w^{\chi}_r-t_r), (w^{\chi}_r-t_r)-\left(\frac{\psi(rx)}{d_r}\chi-t_r\right)\right\}=0,\quad\hbox{on } \RR^n.
\end{equation}
Since we assume that $u\in C^{1+\alpha}(\RR^n)$ and $\varphi\in C^{2s+\alpha}(\RR^n)$, for some $\alpha \in ((2s-1)\vee 1/2,s)$, we see that the function $w^{\chi}_r-t_r$ belongs to $C^{1+\gamma}(\RR^n)$, for some $\gamma>2s-1$. It follows from Proposition \ref{prop:Uniqueness} that the function $w^{\chi}_r-t_r \in C^{1+\gamma}(\RR^n)$ is the unique solution to the obstacle problem \eqref{eq:Obstacle_problem_w_r_cutoff}. From the proof of Proposition \ref{prop:Existence_Holder_obstacle_problem}, we see from estimate \eqref{eq:u_eps_Holder_estimate}, that for all $\beta\in (0,2s-1)$, there is a positive constant $C=C(\beta)$, such that the function $w^{\chi}_r-t_r$ satisfies the Schauder estimate,
\begin{equation*}
\begin{aligned}
\|w^{\chi}_r-t_r\|_{C^{1+\beta}(\RR^n)} &\leq C\left(\|g_r\|_{C(\RR^n)}+\left\|(-\Delta)^s\left(\frac{\psi(rx)}{d_r}\chi\right)+\frac{\psi(rx)}{d_r}\chi\right\|_{C(\RR^n)}\right.\\
&\quad\left.+\left\|\frac{\psi(rx)}{d_r}\chi-t_r\right\|_{C(\RR^n)}\right),\quad\forall r>0.
\end{aligned}
\end{equation*}
By \cite[Proposition 2.9]{Silvestre_2007} and Lemma \ref{lem:Sup_estimate}, it follows that
\begin{equation*}
\|t_r\|_{C^{1+\beta}(\RR^n)} \leq C\|g_r\|_{C(\RR^n)},\quad\forall r>0,
\end{equation*}
and so, the function $w^{\chi}_r$ satisfies the Schauder estimate \eqref{eq:Schauder_estimate_w_r_cutoff}. In Step \ref{step:Holder_continuity}, we use estimate \eqref{eq:Schauder_estimate_w_r_cutoff} to obtain a uniform Schauder estimate on $\|w_r\|_{C^{1+\beta}(B'_{1/2})}$, for all $r>0$ sufficiently small.
\end{step}

\begin{step}[H\"older continuity]
\label{step:Holder_continuity}
Our goal is now to use estimate \eqref{eq:Schauder_estimate_w_r_cutoff} and prove that \eqref{eq:Uniform_boundedness_C_1_alpha_n_ball_rescalings} holds. For this purpose we employ the iteration argument used to prove \cite[Theorem 8.11.1]{Krylov_LecturesHolder}. For all $k \in \NN$, we let
$$
r_k=\frac{1}{4}\sum_{i=0}^k \frac{1}{2^i},\quad\forall k\in\NN,
$$
and we denote $B'_k:=B'_{r_k}$, for brevity. We now let $\chi_k:\RR^n\rightarrow[0,1]$ be a smooth function such that 
$$
\chi_k \equiv 1,\quad\hbox{ on } B'_k,\quad\hbox{and}\quad \chi_k\equiv 0, \quad\hbox{ on } (B'_{k+1})^c.
$$
In addition, we can choose the cut-off functions $\chi_k$, such that there is a positive constant, $C$, satisfying the property, 
\begin{equation}
\label{eq:Fractional_laplacian_cutoff}
\|(-\Delta)^s\chi_k\|_{C(\RR^n)} \leq C 2^{2k},\quad\forall k \in \NN.
\end{equation}
We denote
$$
\alpha_k:=\|v_r\chi_k\|_{C^{1+\beta}(\RR^n)}.
$$
We apply estimate \eqref{eq:Schauder_estimate_w_r_cutoff} to the function $w^{\chi}_r$, where we recall that $w^{\chi}_r=\chi w_r$ by \eqref{eq:Definition_w_r_cutoff}, and we choose $\chi=\chi_k$. We denote 
\begin{equation}
\label{eq:Definition_f_r}
f^k_r(x):=\int_{\RR^n}\frac{(\chi_k(x)-\chi_k(y))(w_r(x)-w_r(y))}{|x-y|^{n+2s}}\, dy,\quad\forall x \in \RR^n,
\end{equation}
for all $r>0$ and $k \in \NN$. Using definition \eqref{eq:Definition_g_r} of the function $g_r$, estimate \eqref{eq:Schauder_estimate_w_r_cutoff} yields
\begin{align*}
\alpha_k &\leq C\left(\|w_r\chi_k\|_{C(\RR^n)} + \|h_r\chi_k\|_{C(\RR^n)} + \frac{1}{d_r}\|\psi(rx)\chi_k\|_{C^{1+\beta}(\RR^n)}\right.\\
&\quad\left. +\left\|(-\Delta)^s\left(\frac{\psi(rx)}{d_r}\chi_k\right)\right\|_{C(\RR^n)}
+\|w_r(-\Delta)^s\chi_k - f^k_r\|_{C(\RR^n)} \right).
\end{align*}
By Lemma \ref{lem:Uniform_boundedness_rescalings}, definitions \eqref{eq:Definition_v_r_modified} of $w_r$ and \eqref{eq:Definition_psi_w_r} of $\psi$, and estimate \eqref{eq:Uniform_bound_h_r_w_r} of $h_r$, we can find positive constants, $C$ and $r_0$ such that
\begin{align*}
\|w_r\chi_k\|_{C(\RR^n)} + \|h_r\chi_k\|_{C(\RR^n)} \leq C,\quad\forall r\in(0,r_0),\ k \in \NN,
\end{align*}
and, using the properties of the cutoff functions $\chi_k$, we obtain
$$
 \frac{1}{d_r}\|\psi(rx)\chi_k\|_{C^{1+\beta}(\RR^n)}+\left\|(-\Delta)^s\left(\frac{\psi(rx)}{d_r}\chi_k\right)\right\|_{C(\RR^n)} \leq C2^{2k},\quad\forall k \in \NN.
$$
It follows that 
\begin{align}
\label{eq:Estimate_alpha_k_1}
\alpha_k &\leq C\left(2^{2k} + \|w_r(-\Delta)^s\chi_k -f^k_r\|_{C(\RR^n)} \right),\quad\forall r\in (0,r_0),\ k\in\NN.
\end{align}
To evaluate the last term in the preceding inequality, we consider two cases depending on whether the point $x$ belongs to $\hbox{supp }\chi_k$ or $x$ belongs to $(\hbox{supp }\chi_k)^c$.

\begin{case}[Points $x\in \hbox{supp }\chi_k$]
\label{case:Estimate_alpha_k_x_less_1}
From definition \eqref{eq:Definition_f_r} of the function $f^k_r$, we obtain that there is a positive constant, $C$, such that
$$
|f^k_r(x)| \leq C 2^{k}\left(1+\|\nabla v_r\|_{C(B'_{k+1})}\right),\quad\forall k \in \NN.
$$
Thus, in the case when $x$ is contained in $\hbox{supp }\chi_k$, the preceding inequality together with estimates \eqref{eq:Uniform_boundedness_rescalings} and  \eqref{eq:Fractional_laplacian_cutoff}, give us
\begin{align*}
|w_r(x)(-\Delta)^s\chi_k(x) -f^k_r(x)|\leq C 2^{2k}(1+\|\nabla v_r\|_{C(B'_{k+1})}),\quad\forall k \in \NN.
\end{align*}
\end{case}

\begin{case}[Points $x\notin \hbox{supp }\chi_k$]
\label{case:Estimate_alpha_k_x_greater_1}
From definition \eqref{eq:Definition_f_r} of $f_r$, we see that
\begin{align*}
w_r(x)(-\Delta)^s\chi_k(x) -f^k_r(x) = \int_{\RR^n}\frac{w_r(y)(\chi_k(x)-\chi_k(y))}{|x-y|^{n+2s}}\, dy.
\end{align*}
Because we assume that $x$ does not belong to $\hbox{supp }\chi_k$, we have that $\chi_k(x)=0$ and $\nabla\chi_k(x)=0$, and so
\begin{align*}
w_r(x)(-\Delta)^s\chi_k(x) -f^k_r(x) = -\int_{B'_1}\frac{w_r(y)(\chi_k(y)-\chi_k(x)-\nabla\chi_k(x)\dotprod(y-x))}{|x-y|^{n+2s}}\, dy.
\end{align*}
By applying Lemma \ref{lem:Uniform_boundedness_rescalings}, we obtain that there are positive constants, $C$ and $r_0$, such that
\begin{align*}
|w_r(x)(-\Delta)^s\chi_k(x) -f^k_r(x)| &\leq C2^{2k},\quad\forall k \in \NN,\quad\forall r\in (0,r_0).
\end{align*}
\end{case}

Combining Cases \ref{case:Estimate_alpha_k_x_less_1} and \ref{case:Estimate_alpha_k_x_greater_1}, we obtain that
\begin{align*}
\|w_r(-\Delta)^s\chi_k -f^k_r\|_{C(\RR^n)}\leq C 2^{2k}(1+\|\nabla v_r\|_{C(B'_{k+1})}),\quad\forall k \in \NN,\quad\forall r\in (0,r_0),
\end{align*}
and so, from inequality \eqref{eq:Estimate_alpha_k_1}, it follows that
\begin{align}
\label{eq:Estimate_alpha_k_2}
\alpha_k &\leq C 2^{2k}\left(1+\|\nabla v_r\|_{C(B'_{k+1})}\right),\quad\forall k\in\NN,\quad\forall r\in (0,r_0).
\end{align}
Applying the Interpolation inequalities \cite[Theorems 3.2.1 \& 8.8.1]{Krylov_LecturesHolder}, we have that there are positive constants, $C$ and $m$, such that for all $\eps>0$, 
$$
\|\nabla v_r\|_{C(B'_{k+1})} \leq \eps \alpha_{k+2} + C\eps^{-m} \|v_r\|_{C(B'_1)},\quad\forall k\in\NN.
$$
Estimate \eqref{eq:Estimate_alpha_k_2} together with the preceding inequality and Lemma \ref{lem:Uniform_boundedness_rescalings}, give us that
\begin{align*}
\alpha_k &\leq C 2^{2k}(\eps^{-m}+\eps \alpha_{k+2}),\quad\forall k\in\NN.
\end{align*}
Because $\eps>0$ is arbitrarily chosen, we redefine it to be $\eps^2 C^{-1}2^{-2k}$. Then the preceding inequality becomes
\begin{align*}
\alpha_k &\leq \eps^2\alpha_{k+2} + C \eps^{-2m} 2^{2(m+1)k},\quad\forall k\in\NN.
\end{align*}
We multiply the inequality by $\eps^k$, and we obtain
\begin{align*}
\eps^k\alpha_k &\leq \eps^{k+2}\alpha_{k+2} + C \eps^{-2m+k} 2^{2(m+1)k},\quad\forall k\in\NN.
\end{align*}
By choosing $\eps \in (0,1)$ sufficiently small such that
$$
\sum_{k=0}^{\infty} \eps^{-2m+k} 2^{2(m+1)k} <\infty,
$$
we obtain that
\begin{align*}
\sum_{k=0}^{\infty}\eps^k\alpha_k &\leq \sum_{k=0}^{\infty}\eps^{k+2}\alpha_{k+2} + C, 
\end{align*}
and so, it follows that $\alpha_0 \leq C$. Estimate \eqref{eq:Uniform_boundedness_C_1_alpha_n_ball_rescalings} now follows.
\end{step}
This concludes the proof.
\end{proof}

The uniform Schauder estimate \eqref{eq:Uniform_boundedness_C_1_alpha_n_ball_rescalings} on $B'_{1/4}$ can now be used to obtain a uniform Schauder estimate in $B^+_{1/8}$ of the rescaling sequence. We have the following consequence of Lemma \ref{lem:Uniform_boundedness_C_1_alpha_n_ball_rescalings}. 

\begin{lem}[Uniform Schauder estimates on $B^+_{1/8}$]
\label{lem:Uniform_boundedness_C_1_alpha_n_plus_1_ball_rescalings}
Assume that the hypotheses of Lemma \ref{lem:Uniform_boundedness_C_1_alpha_n_ball_rescalings} hold. Then there are positive constants, $C$, $\gamma\in (0,1)$ and $r_0$, such that
\begin{equation}
\label{eq:Uniform_boundedness_C_1_alpha_n_plus_1_ball_rescalings}
\begin{aligned}
\|v_r\|_{C^{\gamma}(\bar B^+_{1/8})} &\leq C,\\
\|\partial_{x_i} v_r\|_{C^{\gamma}(\bar B^+_{1/8})} &\leq C,\quad\forall i=1,\ldots,n,\\
\||y|^a\partial_y v_r\|_{C^{\gamma}(\bar B^+_{1/8})} &\leq C,
\end{aligned}
\end{equation}
for all $r\in (0,r_0)$.
\end{lem} 

\begin{proof}
By construction of the rescaling sequence, $\{v_r\}_{r>0}$, and from Lemma \ref{lem:Uniform_boundedness_H_1_rescalings}, we know that the functions $v_r$ belongs to $H^1(B^+_{1/4},|y|^a)$ and solve the equation $L_a v_r=0$ on $B^+_{1/4}$. Moreover, by Lemma 
\ref{lem:Uniform_boundedness_C_1_alpha_n_ball_rescalings}, the function $v_r\upharpoonright_{B'_{1/4}}$ is H\"older continuous. By adapting the proof of \cite[Theorem 2.4.6]{Fabes_Kenig_Serapioni_1982a} to the case of non-zero boundary data, as in \cite[Theorem 8.27]{GilbargTrudinger}, we obtain that the function $v_r \in C^{\gamma}(\bar B^+_{1/8})$, for some positive constant, $\gamma\in (0,1)$. Moreover, because the sequence of functions $\{v_r\}_{r>0}$ satisfies the uniform Schauder estimate \eqref{eq:Uniform_boundedness_C_1_alpha_n_ball_rescalings}, we also obtain that there are positive constants, $C$ and $r_0$, such that
\begin{equation}
\label{eq:Holder_estimate_v_r}
\|v_r\|_{C^{\gamma}(\bar B^+_{1/8})} \leq C,\quad\forall r\in (0,r_0).
\end{equation}
A similar argument can be applied to the derivatives $\partial_{x_i} v_r$, for all $i=1,\ldots,n$. Notice that the function $\partial_{x_i} v_r$ solves the equation $L_a \partial_{x_i} v_r=0$ on $B^+_{1/4}$, and Lemma \ref{lem:Uniform_boundedness_C_1_alpha_n_ball_rescalings} gives that the derivative $\partial_{x_i} v_r\upharpoonright_{B'_{1/4}}$ is a H\"older continuous function. Because the derivative $\partial_{x_i} v_r\upharpoonright_{B^+_{1/4}}$ is a $L_a$-harmonic function, and the boundary condition on $B'_{1/4}$ is H\"older continuous, we can prove that $\partial_{x_i} v_r\in H^1(B^+_{1/4}, |y|^a)$. Therefore, we can conclude that
there are positive constants, $C$, $\gamma$ and $r_0$, such that
\begin{equation*}
\|\partial_{x_i} v_r\|_{C^{\gamma}(\bar B^+_{1/8})} \leq C,\quad \forall i =1,\ldots, n,\quad\forall r\in (0,r_0).
\end{equation*}
We now consider the case of the derivative $\partial_y v_r$. From \cite[\S 2.3]{Caffarelli_Silvestre_2007}, we know that the function $w_r:=|y|^a\partial_y v_r$ solves the conjugate equation, $L_{-a} w_r=0$ on $B^+_{1/4}$, where we recall that $a=1-2s$. From Lemma \ref{lem:Uniform_boundedness_H_1_rescalings}, we know that the function $w_r$ belongs to $L^2(B^+_{1/4}, |y|^{-a})$.  From definitions \eqref{eq:Auxiliary_function_v} of the auxiliary function $v$, and \eqref{eq:Rescaling} of the rescaling $v_r$, we obtain that the boundary condition is given by
$$
\lim_{y\downarrow 0} w_r(x,y)=\frac{r^{2s}}{d_r} \left((-\Delta)^s u(rx)-(-\Delta)^s \varphi(rx)+(-\Delta)^s \varphi(O)\right). 
$$
Because the function $u$ solves the obstacle problem \eqref{eq:Obstacle_problem_u} and $O\in\partial\{u=\varphi\}$, we have that $(-\Delta)^s u(O)=0$. In addition, the functions $u$ and $\varphi$ belong to $C^{1+\alpha}(B'_1)$, for all $\alpha\in (2s-1,s)$, and so, we see that there are positive constants, $\beta\in (0,1)$, $C$ and $r_0$, such that
\begin{equation*}
\|w_r\|_{C^{\beta}(B'_{1/4})} \leq C,\quad\forall r\in (0,r_0).
\end{equation*}
As in the case of the derivatives $\partial_{x_i} v_r$, we can prove that $w_r \in H^1(B^+_{1/4}, |y|^{-a})$. Then again, the argument applied in the case of the rescaling sequence $\{v_r\}_{r>0}$ to prove estimate \eqref{eq:Holder_estimate_v_r}, gives us that there are positive constants, $C$, $\gamma$ and $r_0$, such that
\begin{equation*}
\||y|^a\partial_y v_r\|_{C^{\gamma}(\bar B^+_{1/8})}=\|w_r\|_{C^{\gamma}(\bar B^+_{1/8})} \leq C,\quad\forall r\in (0,r_0).
\end{equation*}
This concludes the proof.
\end{proof}

We can now give the 
\begin{proof}[Proof of Proposition \ref{prop:Phi_at_0}]
As in the proof of \cite[Lemma 6.1]{Caffarelli_Salsa_Silvestre_2008}, we consider two cases, depending on whether condition \eqref{eq:Fraction_d_r_r_power_finite} or condition \eqref{eq:Fraction_d_r_r_power_infty} is satisfied. 

If condition \eqref{eq:Fraction_d_r_r_power_finite} holds, we can apply the same argument as in the proof of \cite[Lemma 6.1]{Caffarelli_Salsa_Silvestre_2008} to obtain that $\Phi^p_v(0+)=n+a+2(1+p)$, and so, identity \eqref{eq:Phi_at_0_p} holds.

Now assume that condition \eqref{eq:Fraction_d_r_r_power_infty} holds. Then, we may assume without loss of generality that
$$
\Phi^p_v(r) = r \frac{d}{dr} \log F_v(r),\quad\forall r \in (0,1).
$$
From Lemma \ref{lem:Uniform_boundedness_C_1_alpha_n_plus_1_ball_rescalings}, we can find a subsequence, $\{v_{r_k}\}_{k>0}$, which converges strongly in the space $H^1(B^+_{1/8},|y|^a)$ to a function $v_0\in H^1(B^+_{1/8},|y|^a)$. From the complementarity conditions \eqref{eq:Upper_bound_L_a} and \eqref{eq:Equality_L_a}, it follows that the rescalings $v_{r_k}$ satisfy 
\begin{align*}
v_{r_k} &\geq 0 \quad\hbox{on } B'_{1/8},\\
L_a v_{r_k} &=0 \quad\hbox{on } B_{1/8}\backslash B'_{1/8},\\
L_a v_{r_k} &\leq \frac{r^{1-a}}{d_r} h(rx)\cH^{n}|_{\{y=0\}} \quad\hbox{on } B_{1/8}, 
\end{align*}
where the function $h$ is defined by \eqref{eq:Definition_h}. Because we assume that $\alpha\in (0,s)$ and $p\in (s,2s-1/2)$, we can choose $\alpha$ close enough to $s$, so that we have $2s+\alpha-1-p>0$. Combining this inequality with condition \eqref{eq:Fraction_d_r_r_power_infty} and estimate \eqref{eq:Growth_h}, it follows that
$$
\frac{r^{1-a}}{d_r} h(rx) \rightarrow 0,\quad\hbox{as } r\downarrow 0,\quad\forall x \in B'_{1/8},
$$
and so, the function $v_0$ satisfies
\begin{align*}
v_0 &\geq 0 \quad\hbox{on } B'_{1/8},\\
v_0(x,y)&=v_0(x,-y) \quad\forall  (x,y)\in B_{1/8}\backslash B'_{1/8},\\
L_a v_0 &=0 \quad\hbox{on } B_{1/8}\backslash \left( B'_{1/8}\cap\{v_0=0\}\right),\\
L_a v_0 &\leq 0 \quad\hbox{on } B_{1/8}.
\end{align*}
By \cite[Lemma 6.1]{Caffarelli_Salsa_Silvestre_2008}, it follows that $\Phi_{v_0}(r) \geq n+a+2(1+s)$, where the function $\Phi_{v_0}(r)$ is defined in \cite[Formula (3.22)]{Caffarelli_Salsa_Silvestre_2008} by
$$
\Phi_{v_0}(r):=r(1+C_0r) \frac{d}{dr} \log\max\{F_{v_0}(r), r^{n+a+4}\},
$$
where $C_0$ is a positive constant. By \cite[Lemma 6.5]{Caffarelli_Salsa_Silvestre_2008}, we obtain that there are positive constants, $C$ and $r_0$, such that $F_{v_0}(r) \leq Cr^{n+a+2(1+s)}$, for all $r\in (0,r_0)$. We now consider two cases. If we have $F_{v_0}(r) \leq r^{n+a+2(1+p)}$, where we recall that we assume that $p\geq s$, then clearly we have that $\Phi^p_{v_0}(0+)=n+a+2(1+p)$, from the definition \eqref{eq:Phi} of the function $\Phi_{v_0}(r)$. If we have that 
$$
r^{n+a+2(1+p)} \leq F_{v_0}(r) \leq Cr^{n+a+2(1+s)},
$$
then it follows from definition \eqref{eq:Phi} of the function $\Phi_{v_0}(r)$, and \cite[Formula (3.22)]{Caffarelli_Salsa_Silvestre_2008} of the function $\Phi_{v_0}(r)$ that
\begin{align}
\Phi^p_{v_0}(r) &= \frac{1}{1+C_0r} \Phi_{v_0}(r)\notag\\
                \label{eq:Expansion_Phi_v_0} 
                & = 2\frac{r\int_{B_r}|\nabla v_0|^2|y|^a}{\int_{\partial B_r}|v_0|^2|y|^a} +n+a.
\end{align}
By letting $r$ tend to $0$ in the first identity from above, we obtain by Proposition \ref{prop:Monotonicity_formula} and \cite[Lemma 6.1]{Caffarelli_Salsa_Silvestre_2008} that 
\begin{equation}
\label{eq:Lower_bound_Phi_v_second_case}
\Phi^p_{v_0}(0+) = \Phi_{v_0}(0+) \geq n+a+2(1+s).
\end{equation}
From the proof of Lemma \ref{lem:Uniform_boundedness_H_1_rescalings}, identity \eqref{eq:Expansion_Phi} and inequality \eqref{eq:Inequality_second_term_Phi} imply that
$$
\Phi^p_{v}(r) =  2\frac{r\int_{B_r}|\nabla v|^2|y|^a}{\int_{\partial B_r}|v|^2|y|^a} +n+a+ O(r^{2(\alpha-p+s-1/2)}).
$$
For all $t,r>0$, we have that
\begin{align*}
\frac{rt\int_{B_{rt}}|\nabla v|^2|y|^a}{\int_{\partial B_{rt}}|v|^2|y|^a}
&= \frac{r\int_{B_r}|\nabla v_t|^2|y|^a}{\int_{\partial B_r}|v_t|^2|y|^a},
\end{align*}
from which it follows that
$$
\Phi^p_{v}(tr) =  2\frac{r\int_{B_r}|\nabla v_t|^2|y|^a}{\int_{\partial B_r}|v_t|^2|y|^a} +n+a+ O((tr)^{2(\alpha-p+s-1/2)}).
$$
By letting $t$ tend to zero, using the strong convergence of the sequence $\{v_{r_k}\}_{k>0}$ to $v_0$ in the $H^1(B^+_{1/8})$ norm, and the fact that $\alpha-p+s-1/2>0$, we obtain that
$$
\Phi^p_{v}(0+) =  2\frac{r\int_{B_r}|\nabla v_0|^2|y|^a}{\int_{\partial B_r}|v_0|^2|y|^a} +n+a,\quad\forall r\in (0,1).
$$
Identity \eqref{eq:Expansion_Phi_v_0}  gives us that 
\begin{equation}
\label{eq:Homogeneity_v_0} 
\Phi^p_{v}(0+) = \Phi^p_{v_0}(r) = \Phi^p_{v_0}(0+),\quad\forall r\in (0,1).
\end{equation}
The preceding identity together with \eqref{eq:Lower_bound_Phi_v_second_case}, gives us that \eqref{eq:Phi_at_0} holds.

This concludes the proof.
\end{proof}

\subsection{Optimal regularity of solutions}
\label{sec:Solutions_optimal_regularity}
In this section we prove the optimal regularity of solutions. In Proposition \ref{prop:Growth_v_around_0}, we prove a growth estimate of the auxiliary function $v$ in a neighborhood of a free boundary point, which we then use to give the optimal regularity of solutions in Lemma \ref{lem:Optimal_regularity}. We conclude with the proof of our main result, Theorem \ref{thm:Solutions}.

Analogous to \cite[Lemma 6.6]{Caffarelli_Salsa_Silvestre_2008}, we have the following consequence of Proposition \ref{prop:Phi_at_0}.
\begin{lem}
\label{lem:Upper_bound_F}
If $\Phi^p_v(r)\rightarrow \mu$, as $r\downarrow 0$, then there are positive constants, $C$ and $r_0$, such that
\begin{equation}
\label{eq:Upper_bound_F} 
F_v(r) \leq C r^{\mu},\quad\forall r\in (0,r_0),
\end{equation}
where the function $F_v$ is defined by \eqref{eq:F}.
\end{lem}

\begin{proof}
To obtain estimate \eqref{eq:Upper_bound_F} we can apply exactly the same argument that was used to prove \cite[Lemma 6.6]{Caffarelli_Salsa_Silvestre_2008}.
\end{proof}

In Proposition \ref{prop:Growth_v_around_0}, we prove an estimate of the growth of the function $v$ in a neighborhood of a free boundary point. The analogue of this result can be found in \cite[Lemma 6.5]{Caffarelli_Salsa_Silvestre_2008}. While the two results  are similar in spirit, their proofs are different. The proof of \cite[Lemma 6.5]{Caffarelli_Salsa_Silvestre_2008} is based on comparison arguments, which we cannot adapt to our case because the singular measure $L_a v$ is nontrivial on the set $\{y=0\}\backslash\{v=0\}$. To overcome this difficulty, in the proof of Proposition \ref{prop:Growth_v_around_0}, we construct a suitable auxiliary function, $\psi$, to compensate the singular measure, and then we apply the Moser iterations method to obtain the growth estimate \eqref{eq:Growth_v_around_O}. 
\begin{prop}[Growth of $v$ in a neighborhood of a free boundary point]
\label{prop:Growth_v_around_0}
Let $s\in (1/2,1)$. Assume that the obstacle function, $\varphi \in C^{2s+\alpha}(\RR^n)$, and $u\in C^{1+\alpha}(\RR^n)$, for all $\alpha\in (0,s)$, and that $u$ is a solution to problem \eqref{eq:Obstacle_problem_simple}. Then there are positive constants, $C$ and $r_0$, such that
\begin{equation}
\label{eq:Growth_v_around_O}
0 \leq v(x,0) \leq C|x|^{1+s},\quad\forall x \in B'_{r_0}.
\end{equation}
\end{prop}

\begin{proof}
We divide the proof into several steps. 

\setcounter{step}{0}
\begin{step}[Inequality satisfied by $v^+$]
\label{step:Inequality_v_+}
We want to show that there is a positive constant, $C_0$, such that the function $v^+$ satisfies the inequality
\begin{equation}
\label{eq:Inequality_v_+}
\int_{B_1} \nabla v^+\nabla\eta|y|^a \leq  C_0 \int_{B_1'}  |x|^{\alpha}\eta(x),
\end{equation}
for all nonnegative test functions, $\eta\in H^1_0(B_1, |y|^a)$. Because we assume that $s\in (1/2,1)$, we have that $a=1-2s \in (-1,0)$. Then any function in $H^1_0(B_1,|y|^a)$ is also contained in $H^1_0(B_1)$, and by \cite[Theorem 5.5.1]{Evans}, it has a well-defined trace on $B'_1$ and on $\partial B_1$ in the $L^2(B'_1)$ and $L^2(\partial B_1)$ sense, respectively. The space $H^1_0(B_1,|y|^a)$ is the closure of the space of smooth functions $C^{\infty}_c(B_1)$ with respect to the norm $\|\cdot\|_{H^1(B_1,|y|^a)}$.

Let $\eps>0$ and let $\phi_{\eps}:\RR\rightarrow [0,1]$ be a smooth function satisfying the properties
\begin{equation}
\label{eq:Properties_psi_eps}
\phi_{\eps}' \geq 0,\quad \phi_{\eps}(t)=0\quad\hbox{ for } t<\eps,\quad \phi_{\eps}(t)=1\quad\hbox{ for } t>2\eps.
\end{equation}
Let $\eta\in H^1_0(B_1,|y|^a)$ be a nonnegative function. Integrating the singular measure $-L_a v$ against the test function $\phi_{\eps}(v^+)\eta$, and using identity \eqref{eq:Equality_L_a}, we obtain 
\begin{align*}
\int_{B_1} \nabla v\nabla (\phi_{\eps}(v^+)\eta) |y|^a = 2\int_{B_1'} \left((-\Delta)^su(x)+(-\Delta)^s\varphi(O)-(-\Delta)^s\varphi(x)\right) \phi_{\eps}(v^+)\eta.
\end{align*}
Because $(-\Delta)^s u(x)=0$, for all $x\in\{u>\varphi\}\cap B'_1$, that is, for all $x\in\{v>0\}\cap B'_1$, we see from properties \eqref{eq:Properties_psi_eps} of the function $\phi_{\eps}$, that the preceding identity becomes
\begin{align*}
\int_{B_1} \nabla v\nabla \eta \phi_{\eps}(v^+)|y|^a + \int_{B_1} \nabla v \nabla v^+\phi'_{\eps}(v^+)\eta |y|^a = 2\int_{B_1'} \left((-\Delta)^s\varphi(O)-(-\Delta)^s\varphi(x)\right) \phi_{\eps}(v^+)\eta.
\end{align*}
Using again properties \eqref{eq:Properties_psi_eps} of the functions $\phi_{\eps}$, and the fact that the function $\eta$ is assumed to be nonnegative, we obtain that 
$$
\int_{B_1} \nabla v \nabla v^+\phi'_{\eps}(v^+)\eta |y|^a \geq 0,
$$
and so, by letting $\eps$ tend to zero, the preceding two formulas give us 
$$
\int_{B_1} \nabla v^+\nabla\eta|y|^a \leq \int_{B_1'} h(x)\eta(x),
$$
for all nonnegative test functions, $\eta\in H^1_0(B_1,|y|^a)$, where the function $h$ is defined in \eqref{eq:Definition_h}. The preceding inequality and \eqref{eq:Growth_h} imply \eqref{eq:Inequality_v_+}. 
\end{step}

\begin{step}[Construction of an auxiliary function]
\label{step:Construction_auxiliary_function_v_+}
We want to build an auxiliary function, $\psi$, to compensate the term appearing on the right-hand side of inequality \eqref{eq:Inequality_v_+}. Our goal in this step is to prove the existence of a function, $\psi \in H^1(B_1, |y|^a)$, such that there is a positive constant, $C_1$, with the properties
\begin{align}
\label{eq:Equation_extension_psi}
L_a\psi(x,y) &= 0,&\quad\forall (x,y)\in \RR^{n+1}\backslash\{y=0\},\\
\label{eq:Fractional_laplacian_psi}
(-\Delta)^s\psi(x,0) &=C_0|x|^{\alpha}\varphi(|x|),&\quad\forall x\in \RR^n,\\
\label{eq:Growth_psi_n+1_dim}
|\psi(x,y)| &\leq C_1 |(x,y)|^{1+s},&\quad\forall (x,y)\in B_1,
\end{align}
where the smooth function with compact support $\varphi:[0,\infty)\rightarrow [0,1]$ is chosen such that $\varphi(t)=1$, for all $t \in (0,1)$. We begin our construction by defining the function $\psi_1$ by (see \cite[p. 76]{Silvestre_2007})
$$
\psi_1(x):= c_{n,-s}\int_{\RR^n}\frac{C_0|z|^{\alpha}\varphi(|z|)}{|x-z|^{n-2s}}\, dz,\quad\forall x\in\RR^n,
$$
and we let $\psi_2(x):=\psi_1(x)-\psi_1(O)$, for all $x \in \RR^n$. Then the function $\psi_2$ is a solution to equation \eqref{eq:Fractional_laplacian_psi}, and by \cite[Proposition 2.8]{Silvestre_2007}, we also have that the function $\psi_2$ belongs to $C^{2s+\alpha}(\RR^n)$, since $|x|^{\alpha}\varphi(|x|)$ is contained in $C^{\alpha}(\RR^n)$. We note that we have $\psi_2(O)=0$, and because $\psi_2$ is a radial function, the gradient $\nabla \psi_2(O)=0$. We see that there is a function, $\psi_0:[0,\infty)\rightarrow\RR$, which belongs to $C^{2s+\alpha}(\bar\RR_+)$ such that $\psi_2(x)=\psi_0(|x|)$, for all $x\in\RR^n$, and $\psi_0(0)=0$ and $\psi'(0)=0$.

 We extend the function $\psi_2$ from $\RR^n$ to $\RR^n\times\RR_+$ such that $\psi_2$ satisfies equation \eqref{eq:Equation_extension_psi}. We let $\psi$ denote the $L_a$-harmonic extension of the function $\psi_2$ from $\RR^n$ to $\RR^n\times\RR_+$. Using the Poisson formula \cite[\S 2.4]{Caffarelli_Silvestre_2007}, the function $\psi$ can be constructed by setting
\begin{equation}
\label{eq:Definition_psi}
\psi(x,y):=\int_{\RR^n} P(z,y) \psi_0(|x-z|) \, dz,\quad\forall (x,y)\in\RR^{n}\times\RR_+,
\end{equation}
where we recall from \cite[Formula (2.3)]{Caffarelli_Silvestre_2007} that the Poisson kernel, $P(x,y)$, is defined by 
\begin{equation}
\label{eq:Poisson_kernel}
P(x,y)= C_{n,s} \frac{y^{2s}}{\left(|x|^2+y^2\right)^{(n+2s)/2}},\quad\forall (x,y)\in \RR^n\times\RR_+,
\end{equation}
where $C_{n,s}$ is a positive constant depending only on $n$ ans $s$. We extend $\psi$ from $\RR^n\times\RR_+$ to $\RR^{n+1}$ by even reflection with respect to the hyperplane $\{y=0\}$. Then the function $\psi$ satisfies conditions \eqref{eq:Equation_extension_psi} and \eqref{eq:Fractional_laplacian_psi}, and it remains to show that $\psi$ satisfies the growth condition \eqref{eq:Growth_psi_n+1_dim}. For this purpose, we prove the following
\begin{claim}[Growth of $\psi$ in the $x$ and $y$ directions]
\label{claim:Growth_psi}
There is a positive constant, $C$, such that
\begin{align}
\label{eq:Growth_x_direction}
|\psi(x,y)-\psi(0,y)| &\leq C|x|^{1+s},\quad\forall (x,y), (0,y) \in B_1,\\
\label{eq:Growth_y_direction}
|\psi(0,y)| &\leq C|y|^{1+s},\quad\forall (0,y)\in B_1.
\end{align}
\end{claim}

\begin{proof}
We prove each inequality in the following two steps.
\setcounter{step}{0}
\begin{step}[Proof of inequality \eqref{eq:Growth_x_direction}]
Taking derivative in the $x_i$ variable in \eqref{eq:Definition_psi}, we obtain
\begin{align*}
\psi_{x_i}(0,y)&=\int_{\RR^n} P(z,y) \psi_0'(|z|) \frac{z_i}{|z|}\, dz,\quad\forall y>0,\quad \forall i=1,\ldots,n.
\end{align*}
Because the function $\psi_0$ and the Poisson kernel $P(\cdot,y)$ are radial functions, we see that $\psi_{x_i}(0,y) =0$, for all $y>0$, and all $i=1,\ldots,n$, which implies that there is a positive constant, $C$, such that
\begin{align*}
|\psi_{x_i}(x,y)| &= |\psi_{x_i}(x,y)- \psi_{x_i}(0,y)|\\
&\leq\int_{\RR^n} P(z,y)|\partial_{x_i}\psi_0(|x-z|)-\partial_{x_i}\psi_0(|z|)|\, dz\\
&\leq\int_{\RR^n} P(z,y)|x|^{2s+\alpha-1}\, dz\\
&\leq |x|^{2s+\alpha-1},\quad\forall (x,y)\in\RR^{n}\times(0,\infty).
\end{align*}
where in the last two inequalities we used the fact that the function $\psi_0$ belongs to $C^{2s+\alpha}(\bar\RR_+)$, and that the Poisson kernel, $P(\cdot,y)$ is a probability density. The preceding inequality immediately implies \eqref{eq:Growth_x_direction}, since we may choose $\alpha>1-s$. 
\end{step}

\begin{step}[Proof of inequality \eqref{eq:Growth_y_direction}]
From the fact that $\psi_0(0)=0$, the definition \eqref{eq:Definition_psi} of the function $\psi$, and \eqref{eq:Poisson_kernel} of the Poisson kernel, $P(x,y)$, we obtain
$$
\frac{\psi(0,y)}{y^{1+s}} = -C_{n,s} y^{s-1}\int_{\RR^n} \frac{|z|^{n+2s}}{\left(|z|^2+y^2\right)^{(n+2s)/2}} \frac{\psi_0(0)-\psi_0(|z|)}{|z|^{n+2s}}\, dz.
$$
Identity \eqref{eq:Fractional_laplacian_psi} evaluated at $x=0$, gives that
$$
\int_{\RR^n} \frac{\psi_0(0)-\psi_0(|z|)}{|z|^{n+2s}}\, dz=0,
$$
from where it follows that
$$
\frac{\psi(0,y)}{y^{1+s}} = C_{n,s} y^{s-1}\int_{\RR^n}\left(1- \frac{|z|^{n+2s}}{\left(|z|^2+y^2\right)^{(n+2s)/2}}\right) \frac{\psi_0(0)-\psi_0(|z|)}{|z|^{n+2s}}\, dz.
$$
To estimate the preceding integral, we split it into the integrals over the sets $\{|z|<2y\}$, $\{2y\leq |z|<1\}$, and $\{1\leq |z|\}$, and we denote them by $I_i$, for $i=1,2,3$, respectively. We estimate each of these integrals separately. To estimate the integral $I_1$, we use the fact that $\psi_0\in C^{2s+\alpha}(\bar\RR_+)$, and that $\psi_0(0)=0$ and $\nabla\psi_0(0)=0$. We can find a positive constant, $C$, such that
$$
|I_1| \leq C y^{s-1} \int_0^{2y} \frac{t^{2s+\alpha}t^{n-1}}{y^{n+2s}}\, dt\leq C y^{s+\alpha-1}.
$$
Because we assume that $\alpha\in (1-s,s)$, we see that the right-hand side in the preceding identity in bounded, for all $y\in (0,1)$.

To estimate the integrals $I_2$ and $I_3$, we can use the Taylor series expansion of the function $(1+y^2/|z|^2)^{-(n+2s)/2}$, since $y/|z| <1$. We have that there is a positive constant, $C$, such that 
\begin{equation}
\label{eq:Taylor_approximation}
\left|1-\frac{|z|^{n+2s}}{\left(|z|^2+y^2\right)^{(n+2s)/2}}\right| \leq C\frac{y^2}{|z|^2},\quad\forall |z|\geq 2y.
\end{equation}
We use the preceding inequality to estimate $I_2$, together with the fact that $\psi_0\in C^{2s+\alpha}(\bar\RR_+)$, and that $\psi_0(0)=0$ and $\nabla\psi_0(0)=0$, and we obtain
$$
|I_2| \leq Cy^{s-1}\int_{2y}^1 \frac{y^2}{t^2} \frac{t^{2s+\alpha} t^{n-1}}{t^{n+2s}}\, dt \leq C y^{1+s}(y^{\alpha-2}-1) \leq C y^{s+\alpha-1}.
$$
As in the preceding case, the right-hand side is bounded, for all $y\in (0,1)$.

To estimate the integral $I_3$, we use the Taylor approximation \eqref{eq:Taylor_approximation} and the fact that the function $\psi_0$ is bounded. Then there is a positive constant, $C$, such that
$$
|I_3| \leq C y^{s-1} \int_{1}^{\infty} \frac{y^2}{t^2} \frac{t^{n-1}}{t^{n+2s}}\, dt \leq C y^{1+s},
$$
since the function $t\mapsto t^{-(3+2s)}$ is integrable on $(1,\infty)$. We obtain that the integral $I_3$ is also bounded, for all $y\in (0,1)$. Thus inequality \eqref{eq:Growth_y_direction} now follows.

\end{step}
This completes the proof of Claim \ref{claim:Growth_psi}.
\end{proof}
Inequalities \eqref{eq:Growth_x_direction} and \eqref{eq:Growth_y_direction} together with the fact that $\psi(O)=0$ imply that the function $\psi$ satisfies condition \eqref{eq:Growth_psi_n+1_dim}.

In conclusion, we constructed a function $\psi\in H^1(B_1, |y|^a)$ which verifies conditions \eqref{eq:Equation_extension_psi}, \eqref{eq:Fractional_laplacian_psi} and \eqref{eq:Growth_psi_n+1_dim}.
\end{step}

\begin{step}[Equality satisfied by $\psi$]
\label{step:Inequality_psi}
We now want to show that the function $\psi$ satisfies the equality
\begin{equation}
\label{eq:Inequality_psi}
\int_{B_1} \nabla \psi\nabla\eta|y|^a = 2C_0 \int_{B_1'} |x|^{\alpha}\eta(x),
\end{equation}
for all test functions, $\eta\in H^1_0(B_1,|y|^a)$. Using identity \eqref{eq:Dirichlet_to_Neumann_map}, together with \eqref{eq:Equation_extension_psi} and \eqref{eq:Fractional_laplacian_psi}, we obtain by integration by parts that
\begin{align*}
\int_{B_1} \nabla \psi\nabla\eta |y|^a &= 2\int_{B_1'} (-\Delta)^s\psi(x,0) \eta\\
&= 2C_0\int_{B_1'} |x|^{\alpha} \varphi(|x|)\eta.
\end{align*}
Because the function $\varphi$ was chosen such that $\varphi(t)=1$, for all $t\in[0,1)$, inequality \eqref{eq:Inequality_psi} now follows.
\end{step}

\begin{step}[Proof of estimate \eqref{eq:Growth_v_around_O}]
We can now prove estimate \eqref{eq:Growth_v_around_O}. Let $\eta\in H^1_0(B_1,|y|^a)$ be a nonnegative test function. From inequalities \eqref{eq:Inequality_v_+} and \eqref{eq:Inequality_psi}, we obtain that
$$
\int_{B_1} \nabla \left(v^+-\psi\right)\nabla\eta|y|^a \leq 0.
$$
Thus the function $v^+-\psi$ is $L_a$-subharmonic on $B_1$, and we can apply \cite[Theorem 2.3.1]{Fabes_Kenig_Serapioni_1982a} to obtain that there is a positive constant, $C$, such that
\begin{equation}
\label{eq:Supremum_estimate_difference_1}
\sup_{B_r} (v^+-\psi) \leq C \left(\frac{1}{r^{n+1+a}}\int_{B_{2r}}\left(|v^+|^2+|\psi|^2\right)|y|^a\right)^{1/2}, \quad\forall r\in (0,1/2).
\end{equation}
Recall that we may apply \cite[Theorem 2.3.1]{Fabes_Kenig_Serapioni_1982a} because the weight $\fw(x,y)=|y|^a$ belongs to the Muckenhoupt $A_2$ class of functions, and the analogue of the standard Sobolev inequality \cite[Inequality (7.26)]{GilbargTrudinger} in the case of our weighted space, $H^1(B_1,|y|^a)$, is proved in \cite[Theorem (1.6)]{Fabes_Kenig_Serapioni_1982a}.

Using definition \eqref{eq:F} of the function $F$, we can rewrite inequality \eqref{eq:Supremum_estimate_difference_1} in the form
\begin{align*}
\sup_{B_r} (v^+-\psi)
&\leq C r^{-(n+1+a)/2} \left(\left(\int_0^{2r} F_v(t)\, dt\right)^{1/2} + \left(\int_0^{2r} F_{\psi}(t)\, dt\right)^{1/2} \right).
\end{align*}
From Proposition \ref{prop:Phi_at_0}, we may apply inequality \eqref{eq:Upper_bound_F} with $\mu=n+a+2(1+s)$, and we obtain that there are positive constants, $C_1$ and $r_0$, such that
$$
F_v(r) \leq C_1 r^{n+a+2(1+s)},\quad\forall r\in (0,r_0),
$$
while inequality \eqref{eq:Growth_psi_n+1_dim} gives us that there is a positive constant, $C_2$, such that
$$
F_{\psi}(r) \leq C_2 r^{n+a+2(1+s)},\quad\forall r\in (0,1).
$$
Combining now the last three inequalities, and using the fact that $v \geq 0$ on $\RR^n\times\{0\}$, we find a positive constant, $C_3$, such that 
\begin{align*}
\sup_{B'_r} (v-\psi) \leq C_3 r^{1+s},\quad\forall r\in(0,r_0). 
\end{align*}
The preceding inequality together with \eqref{eq:Growth_psi_n+1_dim} and the fact that 
$$
\sup_{B'_r} v \leq \sup_{B'_r} (v-\psi) + \sup_{B'_r} \psi,
$$
imply that inequality \eqref{eq:Growth_v_around_O} holds.
\end{step}
This concludes the proof of inequality \eqref{eq:Growth_v_around_O}.
\end{proof}

We next have the following analogue of \cite[Corollary 6.8]{Caffarelli_Salsa_Silvestre_2008}.
\begin{lem}
\label{lem:Optimal_regularity}
Let $s\in (1/2,1)$. Assume that the obstacle function, $\varphi \in C^{2s+\alpha}(\RR^n)$, and $u\in C^{1+\alpha}(\RR^n)$, for all $\alpha\in (0,s)$, and that $u$ is a solution to problem \eqref{eq:Obstacle_problem_simple}. Then the function $v\in C^{1+s}(\RR^n)$, where $v$ is defined in \eqref{eq:Auxiliary_function_v}.
\end{lem}

\begin{rmk}
Because $v \in C^{1+s}(\RR^n)$ and $\varphi \in C^{2s+\alpha}(\RR^n)$, for all $\alpha\in (0,s)$, we obtain the optimal regularity for $u$, that is $u\in C^{1+s}(\RR^n)$.
\end{rmk}

\begin{proof}[Proof of Lemma \ref{lem:Optimal_regularity}]
Recall from Proposition \ref{prop:Solutions_partial_regularity} that the function $v$ belongs to $C^{1+\alpha}(\RR^n)$, for all $\alpha\in (0,s)$. Moreover, the function $v$ solves the obstacle problem 
$$
\min\{(-\Delta)^s v-(-\Delta)^s \varphi, v\}=0\quad\hbox{on }\RR^n.
$$
We denote $\Lambda :=\{v=0\}.$ By \cite[Proposition 2.8]{Silvestre_2007}, it follows that the function $v$ belongs to $C^{1+s}(\RR^n)$ if we show that the function $w:=(-\Delta)^s v$ belongs to $C^{1-s}(\RR^n)$. Because the function $v$ belongs to $C^{1+\alpha}(\RR^n)$, for all $\alpha\in (0,s)$, it follows by \cite[Proposition 2.6]{Silvestre_2007} that we also have $w \in C(\RR^n)$. It remains to consider the H\"older seminorm of $w$. We want to show that there is a positive constant, $C$, such that
\begin{equation}
\label{eq:Holder_seminorm_w}
|w(x_1)-w(x_2)|\leq C|x_1-x_2|^{1-s},\quad\forall x_1,x_2\in\RR^n.
\end{equation}
We consider the following cases.

\setcounter{case}{0}
\begin{case}[Points $x_1,x_2\notin\hbox{int }\Lambda$]
\label{case:Both_points_outside_Lambda}
For all points $x_1,x_2\notin\Lambda$, we have that $w(x_i)=-(-\Delta)^s\varphi(x_i)$, for $i=1,2.$ Because the function $\varphi$ belongs to $C^{2s+\alpha}(\RR^n)$, for all $\alpha \in (0,s)$, and $s\in (1/2,1)$, we see that $(-\Delta)^s\varphi \in C^{1-s}(\RR^n)$, and so
$$
|w(x_1)-w(x_2)|\leq \left[(-\Delta)^s\varphi\right]_{C^{1-s}(\RR^n)}|x_1-x_2|^{1-s}.
$$
Because $w$ is a continuous function, the preceding inequality holds if $x_1$ and/or $x_2$ belong to $\partial\Lambda$. Therefore, we obtain that
$$
|w(x_1)-w(x_2)|\leq \left[(-\Delta)^s\varphi\right]_{C^{1-s}(\RR^n)}|x_1-x_2|^{1-s},\quad\forall x_1,x_2\notin\hbox{int }\Lambda.
$$
\end{case}

\begin{case}[Points $x_1\in\Lambda$ and $x_2\in\partial\Lambda$]
\label{case:One_point_in_Lambda_one_point_on_boundary}
Let $x_1\in\Lambda$ and $x_2\in\partial\Lambda$. Without loss of generality, we may assume that $|x_1-x_2|=\dist(x_1,\partial\Lambda)$. Then we have that $v(x_i)=0$, for $i=1,2$. We let $\rho:=|x_1-x_2|$. Because $x_2\in\partial\Lambda$, we have by Proposition \ref{prop:Growth_v_around_0} that
\begin{equation}
\label{eq:Inequality_v_B_rho_x_2}
|v(y)| \leq C|x_2-y|^{1+s},\quad\forall y \in B'_{r_0}(x_2).
\end{equation}
Without loss of generality we may assume that $\rho<r_0/2$. Using the fact that $v(x_i)=0$, for $i=1,2$, we have
\begin{align*}
w(x_1)-w(x_2)&=-\int\frac{v(y)}{|x_1-y|^{n+2s}}\, dy + \int\frac{v(y)}{|x_2-y|^{n+2s}}\, dy\\
&=-\int_{B'_{\rho}(x_2)}\frac{v(y)}{|x_1-y|^{n+2s}}\, dy +\int_{B'_{\rho}(x_2)}\frac{v(y)}{|x_2-y|^{n+2s}}\, dy\\
&\quad+\int_{\RR^n\backslash B'_{\rho}(x_2)}v(y)\left(\frac{1}{|x_2-y|^{n+2s}}-\frac{1}{|x_1-y|^{n+2s}}\right)\, dy,
\end{align*}
which we write as a sum of three terms, $I_1$, $I_2$ and $I_3$, respectively. Using inequality \eqref{eq:Inequality_v_B_rho_x_2} and the fact that $|x_1-y|\geq |x_2-y|$, for all $y\in B'_{\rho}(x_2)$, we also have that
$$
|v(y)| \leq C|x_1-y|^{1+s},\quad\forall y \in B'_{\rho}(x_2).
$$
Then we can estimate the integral $I_1$ in the following way
\begin{align*}
|I_1| &\leq C\int_{B'_{\rho}(x_2)} \frac{|x_1-y|^{1+s}}{|x_1-y|^{n+2s}}\, dy\\
& \leq C \rho^{1-s}\\
& = C |x_1-x_2|^{1-s} 
\end{align*}
We can use inequality \eqref{eq:Inequality_v_B_rho_x_2} and the preceding argument to find that 
$$
|I_2| \leq C |x_1-x_2|^{1-s}.
$$
To estimate integral $I_3$, we use the Mean Value theorem and, for each $y\in \RR^n$, there is a point, $x_y \in\RR^n$, on the line connecting the points $x_1$ and $x_2$ such that 
\begin{align*}
|I_3| &\leq \int_{\RR^n\backslash B'_{\rho}(x_2)}|v(y)| \frac{|x_1-x_2|}{|x_y-y|^{n+2s+1}}\, dy\\
&= |x_1-x_2|\left(\int_{B'_{r_0}(x_2)\backslash B'_{\rho}(x_2)} \frac{|v(y)|}{|x_y-y|^{n+2s+1}}\, dy
+\int_{\RR^n\backslash B'_{r_0}(x_2)}\frac{|v(y)|}{|x_y-y|^{n+2s+1}}\, dy\right)\\
&=|x_1-x_2|\left(I_3'+I_3''\right).
\end{align*}
We denote by $I'_3$ and $I''_3$ the two integrals in the parenthesis on the right-hand side of the preceding inequality. We notice that there is a positive constant, $C$, such that $|x_y-y|\geq C|x_2-y|$ on $(\RR^n\backslash B'_{\rho}(x_2))\cap \Lambda^c$, since the vector $x_y$ is a convex combination of $x_1$ and $x_2$, and we recall that $\rho=|x_1-x_2|$. Using inequality \eqref{eq:Inequality_v_B_rho_x_2}, we obtain that there is a positive constant, $C$, such that
\begin{align*}
I_3' &\leq C\int_{B'_{r_0}(x_2)\backslash B'_{\rho}(x_2)} \frac{|x_2-y|^{1+s}}{|x_2-y|^{n+2s+1}}\, dy\\
&\leq C(1+|x_1-x_2|^{-s})\quad\hbox{(recall that $\rho=|x_1-x_2|$ and $\rho \leq r_0/2$)}.
\end{align*}
Using the boundedness of the function $v$, we obtain that there is a positive constant, $C=C(r_0,s)$, such that
\begin{align*}
I_3'' \leq C\|v\|_{C(\RR^n)}\int_{r_0}^{\infty} t^{-2s-2}\, dt \leq C.
\end{align*}
Therefore, combining the estimates of the integrals $I'_3$ and $I''_3$, we obtain
\begin{align*}
|I_3| &\leq C|x_1-x_2|(|x_1-x_2|^{-s}+C)\\
&\leq C|x_1-x_2|^{1-s},
\end{align*}
when $|x_1-x_2|$ is small enough. Using the estimates for the integrals $I_1$, $I_2$, and $I_3$, it follows that
obtain that
$$
|w(x_1)-w(x_2)|\leq C|x_1-x_2|^{1-s},
$$
for all $x_1\in\Lambda$ and $x_2\in\partial\Lambda$, such that $|x_1-x_2|=\dist(x_1,\partial\Lambda)$.
\end{case}

\begin{case}[Points $x_1,x_2\in\Lambda$]
\label{case:Both_points_in_Lambda}
For $i=1,2$, let $x_i\in \Lambda$, and let $x_i'\in\partial\Lambda$ be a projection of the point $x_i$ on the set $\partial\Lambda$, in the sense that $|x_i-x'_i|=\dist(x_i,\partial\Lambda)$. We consider two subcases.
In the \emph{first} case, we assume that
$$
|x_1-x_1'| \leq 5 |x_1-x_2|\quad\hbox{or}\quad |x_2-x_2'| \leq 5 |x_1-x_2|.
$$ 
Without loss of generality, we may assume that the first of the preceding inequalities holds. Then we also have that
$$
|x_2-x_1'| \leq 6 |x_1-x_2|,
$$
which gives us that $|x_2-x_2'| \leq 6 |x_1-x_2|$, and so $|x_2'-x_1'| \leq 12 |x_1-x_2|$. We easily obtain
\begin{align*}
|w(x_1)-w(x_2)|&\leq |w(x_1)-w(x_1')|+|w(x_1')-w(x_2')|+|w(x_2')-w(x_2)| \\
&\leq C|x_1-x_2|^{1-s}\quad\hbox{(using Cases \ref{case:Both_points_outside_Lambda} and \ref{case:One_point_in_Lambda_one_point_on_boundary})}.
\end{align*}
We now consider the \emph{second} case. We assume that 
\begin{equation}
\label{eq:Ineq_point_projection}
|x_i-x_i'| > 5 |x_1-x_2|,\quad i=1,2.
\end{equation}
As before, let $\rho=|x_1-x_2|$ and let $x:=(x_1+x_2)/2$. Using assumption \eqref{eq:Ineq_point_projection}, we see that $v\equiv0$ on $B'_{4\rho}(x)$, and so, we can write
\begin{align*}
w(x_1)-w(x_2) = \int_{\RR^n\backslash B'_{4\rho}(x)} v(y) \left(\frac{1}{|x_2-y|^{n+2s}}-\frac{1}{|x_1-y|^{n+2s}}\right)\, dy
\end{align*}
We estimate the term $|w(x_1)-w(x_2)|$ using the same argument that we applied in the case of the integral $I_3$ in Case \ref{case:One_point_in_Lambda_one_point_on_boundary}. By the Mean Value theorem, we have that
\begin{align*}
|w(x_1)-w(x_2)| &\leq |x_1-x_2|\left(\int_{B'_{r_0}(x)\backslash B'_{4\rho}(x)}  \frac{|v(y)|}{|x_y-y|^{n+2s+1}}\, dy\right.\\
&\quad\left.+\int_{\RR^n\backslash B'_{r_0}(x)} \frac{|v(y)|}{|x_y-y|^{n+2s+1}}\, dy\right),
\end{align*}
where the point $x_y$ is chosen as in Case \ref{case:One_point_in_Lambda_one_point_on_boundary}. We may assume without loss of generality that $r_0\geq 8\rho$. We again use the fact that the vector $x_y$ is a convex combination of $x_1$ and $x_2$, and so we can find a positive constant, $C$, such that $|x_y-y|\geq C|x-y|$ on $B'_{r_0}(x)\backslash B'_{4\rho}(x)$. Using Proposition \ref{prop:Growth_v_around_0}, and the fact that the function $v\equiv 0$ on $B'_{4\rho}(x)$, we have that
$$
|v(y)| \leq C|x-y|^{1+s},\quad\forall y \in B'_{r_0}(x)\backslash B'_{4\rho}(x),
$$
and so, we obtain that 
\begin{align*}
|w(x_1)-w(x_2)| &\leq C|x_1-x_2|(|x_1-x_2|^{-s}+C)\\
&\leq C|x_1-x_2|^{1-s},
\end{align*}
when $|x_1-x_2|$ is small enough.
\end{case}
Thus inequality \eqref{eq:Holder_seminorm_w} follows by combining the preceding cases.
\end{proof}

We can now give the
\begin{proof}[Proof of Theorem \ref{thm:Solutions}]
The existence of H\"older continuous solutions follows from Proposition \ref{prop:Solutions_partial_regularity}. Lemma \ref{lem:Optimal_regularity} gives us the optimal regularity of solutions, $u\in C^{1+s}(\RR^n)$. Proposition \ref{prop:Uniqueness} yields the uniqueness of solutions.
\end{proof}

\appendix
\section{Auxiliary results and proofs}
\label{sec:Auxiliary_results_proofs}

The following lemma is used in the proof of Proposition \ref{prop:Solutions_partial_regularity} to show that the function $g$ defined in \eqref{eq:Definition_g} belongs to the space of functions $C^{2s+\alpha_0}(\RR^n)$, when $u$ belongs to $C^{1+\alpha_0}(\RR^n)$ and $\alpha_0\in (0,1)$.

\begin{lem}
\label{lem:Regularity_g}
Let $s\in(1/2,1)$. Assume that $u\in C^1(\RR^n)$ and $\eta$ is a smooth cut-off function defined as in \eqref{eq:Definition_eta}. Then the function, $h$, defined by
$$
\RR^n \ni x \mapsto h(x):=\int_{\RR^n} \frac{(\eta(y)-\eta(x))(u(y)-u(x))}{|x-y|^{n+2s}}\ dy,
$$
belongs to $C^{2(1-s)}(\RR^n)$.
\end{lem}

\begin{rmk}
Notice that $2(1-s)\in (0,1)$, from our assumption that $s\in (1/2,1)$.
\end{rmk}

\begin{proof}[Proof of Lemma \ref{lem:Regularity_g}]
It is easy to see that the function $h$ is bounded, and so, we only need to consider its H\"older seminorm. For all $x_1,x_2 \in\RR^n$, and find
\begin{align*}
h(x_1)-h(x_2) &= \int_{\RR^n} \frac{(\eta(x_1-y)-\eta(x_1))(u(x_1-y)-u(x_1)- u(x_2-y)+u(x_2))}{|y|^{n+2s}}\ dy\\
&\quad+ \int_{\RR^n} \frac{(u(x_2-y)-u(x_2))(\eta(x_2-y)-\eta(x_2)- \eta(x_1-y)+\eta(x_1))}{|y|^{n+2s}}\ dy
\end{align*}
We evaluate the first integral on $B'_{|x_1-x_2|}(x_1)$ using the inequalities, 
\begin{align*}
|\eta(x_1-y)-\eta(x_1)| &\leq \|\nabla\eta\|_{C(\RR^n)} |y|,\quad\forall y\in\RR^n,\\
|u(x_i-y)-u(x_i)| &\leq \|\nabla u\|_{C(\RR^n)} |y|,\quad i=1,2, \quad\forall y\in\RR^n,
\end{align*}
and then we evaluate the first integral on $(B'_{|x_1-x_2|}(x_1))^c$ using the inequalities
\begin{align*}
|\eta(x_1-y)-\eta(x_1)| &\leq \|\nabla\eta\|_{C(\RR^n)} |y|, \quad\forall y\in\RR^n,\\
|u(x_1+y)-u(x_2+y)| &\leq \|\nabla u\|_{C(\RR^n)} |x_1-x_2|, \quad\forall y\in\RR^n.
\end{align*}
Applying the same reasoning to the second integral, and using the fact that $\eta$ has compact support in $\RR^n$, we obtain
$$
|h(x_1)-h(x_2)| \leq C|x_1-x_2|^{2(1-s)}.
$$
Thus it follows that the function $h$ belongs to $C^{2(1-s)}(\RR^n)$.
\end{proof}

In order to establish the validity of Proposition \ref{prop:Monotonicity_formula}, we first need to prove Lemmas \ref{lem:Estimate_mean_v}, \ref{lem:L_2_estimate_v_boundary}, \ref{lem:L_2_estimate_v} and \ref{lem:Gradient_estimate} which have their analogues in \cite[Lemma 2.9, Lemma 2.13, Corollary 2.15 and Lemma 7.8]{Caffarelli_Salsa_Silvestre_2008}, respectively. The difference from the results in \cite{Caffarelli_Salsa_Silvestre_2008} and our framework appear on the right-hand side of the estimates \eqref{eq:Estimate_mean_v}, \eqref{eq:L_2_estimate_v_boundary} and \eqref{eq:L_2_estimate_v} which need to take into account the fact that we can only assume that the obstacle function $\varphi \in C^{2s+\alpha}(\RR^n)$, for all $\alpha\in (0,s)$, as opposed to $\varphi\in C^2(\RR^n)$ in \cite{Caffarelli_Salsa_Silvestre_2008}. Even though the proofs of the following lemmas are very similar to the ones in \cite{Caffarelli_Salsa_Silvestre_2008} we include them to emphasize the differences. We begin with the analogue of \cite[Lemma 2.9]{Caffarelli_Salsa_Silvestre_2008}.

\begin{lem}[Estimate on the mean of $v$ on $\partial B_r$]
\label{lem:Estimate_mean_v}
Let $s\in (1/2,1)$. Assume that the function $v$ is defined by \eqref{eq:Auxiliary_function_v}. Then we have that
\begin{equation}
\label{eq:Estimate_mean_v}
 \frac{1}{\omega_{n+a}r^{n+a}} \int_{\partial B_r} v(x,y)|y|^a \leq C r^{2s+\alpha},\quad\forall r\in(0,1),
\end{equation}
where $\omega_{n+a}:=\int_{\partial B_1} |y|^a$.
\end{lem}

\begin{proof}
Let $r\in (0,1)$. As in the proof of \cite[Lemma 2.9]{Caffarelli_Salsa_Silvestre_2008}, we consider the auxiliary function
$$
\Gamma(x,y):=\frac{1}{(n+a-1)\omega_{n+a}}\max\{|(x,y)|^{-(n+a-1)}-r^{-(n+a-1)}, 0\},
$$
and the integration by parts formula gives us
\begin{align*}
-\int_{B_1} L_a v(x,y)\Gamma(x,y) = v(O) - \frac{1}{\omega_{n+a}r^{n+a}}  \int_{\partial B_r} v(x,y)|y|^a.
\end{align*}
Using the fact that $v(O)=0$, inequality \eqref{eq:Upper_bound_L_a} together with \eqref{eq:Definition_h} and \eqref{eq:Growth_h}, we obtain that
\begin{align*}
\frac{1}{\omega_{n+a}r^{n+a}}  \int_{\partial B_r} v(x,y)|y|^a &\leq C\int_{B_1} L_a v(x,y)\Gamma(x,y) \\
&\leq C r^{\alpha+2s}\quad\hbox{(where we used the fact that $a=1-2s$),}
\end{align*}
from where estimate \eqref{eq:Estimate_mean_v} now follows.
\end{proof}

We use Lemma \ref{lem:Estimate_mean_v} to prove

\begin{lem}[Estimate of the $L^2(\partial B_r, |y|^a)$ norm of $v$]
\label{lem:L_2_estimate_v_boundary}
Let $s\in (1/2,1)$. Assume that the function $v$ is defined by \eqref{eq:Auxiliary_function_v}. Then there is a positive constant, $C$, such that
\begin{equation}
\label{eq:L_2_estimate_v_boundary}
\int_{\partial B_r} |v|^2 |y|^a \leq C r\int_{B_r}|\nabla v|^2 |y|^a + C r^{n+a+2(2s+\alpha)},\quad\forall r\in (0,1).
\end{equation}
\end{lem}

\begin{rmk}
We notice that the bound $r^{n+a+2(2s+\alpha)}$ on the right-hand side of estimate \eqref{eq:L_2_estimate_v_boundary} is an improvement over the bound $r^{n+a+2(1+\alpha)}$ which follows from the fact that $v\in C^{1+\alpha}(\RR^n)$. We recall that $2s>1$. This improvement is crucial in our proof of the Monotonicity formula in Proposition \ref{prop:Monotonicity_formula}.
\end{rmk}

\begin{proof}[Proof of Lemma \ref{lem:L_2_estimate_v_boundary}]
By Poincar\'e inequality \cite[Lemma 2.10]{Caffarelli_Salsa_Silvestre_2008}, we have
$$
\int_{\partial B_r} |v-k|^2 |y|^a \leq C r\int_{B_r} |\nabla v|^2|y|^a,
$$
where $k$ is given by
$$
k:=\frac{1}{\omega_{n+a}r^{n+a}}\int_{\partial B_r} v|y|^a,
$$
and we recall that $\omega_{n+a}=\int_{\partial B_1} |y|^a$. The preceding inequality gives us that
$$
\int_{\partial B_r} |v|^2 |y|^a \leq 2 C r\int_{B_r}|\nabla v|^2 |y|^a + 2 k^2 r^{n+a},
$$
and, using estimate \eqref{eq:Estimate_mean_v}, inequality \eqref{eq:L_2_estimate_v_boundary} follows immediately.
\end{proof}

Integrating estimate \eqref{eq:L_2_estimate_v_boundary} in the radial direction, we obtain the straightforward analogue of \cite[Corollary 2.15]{Caffarelli_Salsa_Silvestre_2008}. 

\begin{lem}[Estimate of the $L^2(B_r, |y|^a)$ norm of $v$]
\label{lem:L_2_estimate_v}
Let $s\in (1/2,1)$. Assume that the function $v$ is defined by \eqref{eq:Auxiliary_function_v}. Then there is a positive constant, $C$, such that
\begin{equation}
\label{eq:L_2_estimate_v}
\int_{B_r} |v|^2 |y|^a \leq C r^2\int_{B_r}|\nabla v|^2 |y|^a + Cr^{n+a+2(2s+\alpha)+1},\quad\forall r\in (0,1).
\end{equation}
\end{lem}

We include for completion the analogue of \cite[Lemma 7.8]{Caffarelli_Salsa_Silvestre_2008} adapted to our framework. We omit the proof because it is exactly as in \cite{Caffarelli_Salsa_Silvestre_2008}.

\begin{lem}
\label{lem:Gradient_estimate}
Let $s\in (1/2,1)$. Assume that the function $v$ is defined by \eqref{eq:Auxiliary_function_v}. Then the following identity holds
\begin{equation}
\label{eq:Gradient_estimate}
r\int_{\partial B_r} \left(|v_{\tau}|^2-|v_{\nu}|^2\right)|y|^a 
= (n+a-1) \int_{B_r} |\nabla v|^2|y|^a  -2\int_{B_r} (x,y)\dotprod \nabla v L_a v,
\end{equation}
where $v_{\tau}$ is the gradient in the tangential direction to $\partial B_r$, and $v_{\nu}$ is the derivative in the normal direction to $\partial B_r$.
\end{lem}

Using Lemmas \ref{lem:L_2_estimate_v_boundary}, \ref{lem:L_2_estimate_v} and \ref{lem:Gradient_estimate}, we can now give the
\begin{proof}[Proof of Proposition \ref{prop:Monotonicity_formula}]
To prove Proposition \ref{prop:Monotonicity_formula} we follow the proof of \cite[Theorem 3.1]{Caffarelli_Salsa_Silvestre_2008} which we adapt to our framework. Let $\alpha\in (2s-1,s)$ and $p \in [s,\alpha+s-1/2)$. Without loss of generality, we may assume that
\begin{equation}
\label{eq:Ineq_F_v_r}
F_v(r) > r^{(n+a)+2(1+p)}.
\end{equation}
Then we obtain
\begin{align}
\label{eq:Identity_Phi_p}
\Phi^p_v(r) = r\frac{F'(r)}{F(r)}= r\frac{2\int_{\partial B_r} v v_{\nu}|y|^a}{\int_{\partial B_r} |v|^2 |y|^a} + (n+a),
\end{align}
where $v_{\nu}$ is the gradient in the outer normal direction to $\partial B_r$. We let $\Psi^p_v(r):=\Phi^p_v(r)-(n+a)$. Similarly to \cite[Identity (7.11)]{Caffarelli_Salsa_Silvestre_2008}, we have 
\begin{align}
\label{eq:Derivative_phi_p}
\frac{d}{dr}\log \Psi^p_v(r) &= \frac{1}{r} +\frac{\frac{d}{dr} \int_{\partial B_r} v v_{\nu}|y|^a}{\int_{\partial B_r} v v_{\nu}|y|^a}
-\frac{2\int_{\partial B_r} v v_{\nu}|y|^a}{\int_{\partial B_r} |v|^2|y|^a}- \frac{n+a}{r}.
\end{align}
Integration by parts gives the identities
\begin{align}
\label{eq:v_v_nu}
\int_{\partial B_r} v v_{\nu}|y|^a &=\int_{B_r} |\nabla v|^2 |y|^a + \int_{B_r} v L_a v,\quad\hbox{(\cite[Identity (7.12)]{Caffarelli_Salsa_Silvestre_2008})}\\
\label{eq:derivative_v_v_nu}
\frac{d}{dr}\int_{\partial B_r} v v_{\nu}|y|^a &=\int_{\partial B_r} |\nabla v|^2 |y|^a + \int_{\partial B_r} v L_a v. \quad\hbox{(\cite[Identity (7.13)]{Caffarelli_Salsa_Silvestre_2008})}
\end{align}
From Lemma \ref{lem:Gradient_estimate}, we obtain
\begin{align*}
\int_{\partial B_r} |\nabla v|^2 |y|^a 
&= 2\int_{\partial B_r} |v_{\nu}|^2 |y|^a + \frac{n+a-1}{r} \int_{B_r} |\nabla v|^2 |y|^a -\frac{2}{r} \int_{B_r} (x,y)\dotprod \nabla v L_a v,\\
\end{align*}
and identity  \eqref{eq:v_v_nu} gives us that
\begin{align*}
\int_{\partial B_r} |\nabla v|^2 |y|^a 
&= 2\int_{\partial B_r} |v_{\nu}|^2 |y|^a + \frac{n+a-1}{r} \int_{\partial B_r} v v_{\nu} |y|^a \\
&\quad - \frac{n+a-1}{r} \int_{B_r} v L_a v -\frac{2}{r} \int_{B_r} (x,y)\dotprod \nabla v L_a v.
\end{align*}
Using the preceding identity together with \eqref{eq:derivative_v_v_nu} in \eqref{eq:Derivative_phi_p}, it follows that
\begin{align*}
\frac{d}{dr}\log \Psi^p_v(r)  
&=2\left[\frac{\int_{\partial B_r} |v_{\nu}|^2|y|^a}{\int_{\partial B_r} v v_{\nu}|y|^a}-\frac{\int_{\partial B_r} v v_{\nu}|y|^a}{\int_{\partial B_r} |v|^2|y|^a}\right]\\
&\quad-\frac{\frac{n+a-1}{r} \int_{B_r} v L_a v +\frac{2}{r} \int_{B_r} (x,y)\dotprod \nabla v L_a v-\int_{\partial B_r} v L_a v}{\int_{\partial B_r} v v_{\nu}|y|^a}.
\end{align*}
We write the preceding expression as a sum $R(r)+S(r)$. By Schwartz inequality we see that $R(r)\geq 0$ . We now consider $S(r)$, and we want to show that there are positive constants, $C$ and $r_0\in (0,1)$, such that
\begin{equation}
\label{eq:Sup_estimate_S}
|S(r)| \leq C r^{2(\alpha+s-p-1)},\quad\forall r\in (0,r_0).
\end{equation}
We estimate the denominator in the expression of $S(r)$ using identity \eqref{eq:v_v_nu}. By Lemma \ref{lem:L_2_estimate_v_boundary}, inequality \eqref{eq:Ineq_F_v_r}, and the fact that we assume $p<s+\alpha-1/2$, we obtain that there is a constant, $r_0\in (0,1)$, such that
\begin{align*}
\int_{B_r} |\nabla v|^2|y|^a &\geq C(r^{n+a+2(1+p)-1}-r^{n+a+2(2s+\alpha)-1})\\
&\geq Cr^{n+a+2(1+p)-1},\quad\forall r\in(0,r_0).
\end{align*}
Using definition \eqref{eq:Definition_h}, and inequalities \eqref{eq:Growth_h}, \eqref{eq:Properties_v_1}, \eqref{eq:Properties_v_2}, and \eqref{eq:Ineq_v_on_R_n}, we obtain
\begin{align*}
\int_{B_r} v L_a v = \int_{B'_r} v h\leq Cr^{1+2\alpha+n},\quad\forall r\in(0,1).
\end{align*} 
The preceding two inequalities together with identity \eqref{eq:v_v_nu} and the fact that $p<\alpha+s-1/2$  yield
\begin{equation}
\label{eq:Denominator_S}
\int_{\partial B_r} v v_{\nu}|y|^a \geq C r^{n+a+2(1+p)-1},\quad\forall r\in(0,r_0).
\end{equation}
Now we estimate the terms in the numerator of the expression of $S(r)$. Using \eqref{eq:Obstacle_problem_simple}, definition \eqref{eq:Auxiliary_function_v} of $v$ and estimate \eqref{eq:Ineq_v_on_R_n}, definition \eqref{eq:Definition_h} of $h$ and estimate \eqref{eq:Growth_h}, we obtain
\begin{equation}
\label{eq:Numerator_S}
\begin{aligned}
\left|\frac{1}{r} \int_{B_r} v L_a v\right| &\leq Cr^{2\alpha+n},\\
\left|\frac{1}{r} \int_{B_r} (x,y)\dotprod \nabla v L_a v\right| &\leq Cr^{2\alpha+n},\\
\left|\int_{\partial B_r} v L_a v\right| &\leq Cr^{2\alpha+n}.
\end{aligned}
\end{equation}
From estimates \eqref{eq:Denominator_S} and \eqref{eq:Numerator_S}, we obtain \eqref{eq:Sup_estimate_S}.

Because we assume that$p<\alpha+s-1/2$, we see that $2(\alpha+s-p-1) >-1$, and so the lower bound of $|S(r)|$ in \eqref{eq:Sup_estimate_S} is an integrable function of $r$. Because $R(r) \geq 0$, we obtain
$$
\frac{d}{dr}\log\Psi^p_v(r) \geq -C r^{2(\alpha+s-p-1)},\quad\forall r\in (0,r_0),
$$
and so, it follows that
$$
\frac{d}{dr}\left(\log\Psi^p_v(r)+C r^{2(\alpha+s-p)-1}\right) \geq 0,\quad\forall r\in (0,r_0).
$$
Thus the function 
$$
r\mapsto \log\Psi^p_v(r)+C r^{2(\alpha+s-1/2-p)}
$$
is non-decreasing on $(0,r_0)$, from where the conclusion follows immediately using the fact that $\Psi^p_v(r)=\Phi^p_v(r)-(n+a)$.
\end{proof}

%
%


\end{document}